\documentclass[11pt]{amsart}
\usepackage[margin=1in]{geometry}
\usepackage{amssymb,amsmath,amsthm,mathtools}
\usepackage{enumitem}
\usepackage{booktabs}
\usepackage{array}
\usepackage{graphicx}
\usepackage[hidelinks]{hyperref}
\hypersetup{pdftitle={A proof of Wang's planar D-coloring conjecture},pdfauthor={Harshit Verma}}

\newtheorem{theorem}{Theorem}[section]
\newtheorem{lemma}[theorem]{Lemma}
\newtheorem{corollary}[theorem]{Corollary}
\theoremstyle{definition}

\newtheorem{remark}[theorem]{Remark}
\numberwithin{equation}{section}

\newcommand{\chiD}{\chi'_{\mathrm D}}
\newcommand{\Pal}{\mathcal{C}}
\DeclareMathOperator{\Bk}{B}

\setlist[enumerate]{label=\textup{(\roman*)},leftmargin=*}

\begin{document}

\title{A proof of Wang's planar $D$-coloring conjecture}
\author{Harshit Verma}
\address{Yale University, New Haven, CT, USA}
\email{harshit.verma@yale.edu}

\begin{abstract}
We prove Wang's planar $D$-coloring conjecture.  A $D$-coloring of a graph is a proper edge
coloring in which every subgraph isomorphic to $K_4$ minus an edge is rainbow.  Wang conjectured
that every planar graph of maximum degree $\Delta$ has a $D$-coloring with $9$ colors if
$\Delta=4$, $10$ colors if $\Delta=5$, and $2\Delta-1$ colors if $\Delta\ge6$.  Hu, Kong and Wang
proved the cases $\Delta\le5$ and $\Delta\ge33$.  We prove the sharp bound $2\Delta-1$ for every
$\Delta\ge6$, settling the remaining cases.  The proof uses discharging; the case $\Delta=6$ is
computer-assisted.
\end{abstract}

\maketitle

\section{Introduction}

A \emph{diamond} is the graph obtained from $K_4$ by deleting one edge.  Its two vertices of
degree three form its \emph{spine}, and the other two vertices are its \emph{tips}.  A
\emph{$D$-coloring} of a graph is a proper edge coloring in which the five edges of every diamond
subgraph receive distinct colors.  The diamond need not be induced; in particular all six edges
of a $K_4$ receive distinct colors.  Write $\chiD(G)$ for the least number of colors in a
$D$-coloring of $G$.

Wang introduced this parameter in \cite{Wang} and gave a greedy upper bound for general graphs,
which Tian and Wang \cite{TW} improved for large maximum degree.  For planar graphs Wang proposed
the following bounds \cite[Conjecture~4.1]{Wang}: nine colors when the maximum degree is four,
ten when it is five, and $2\Delta(G)-1$ when it is at least six.  Hu, Kong and Wang \cite{HKW} proved the bounds for
maximum degree at most five and for maximum degree at least $33$.  We prove all remaining cases.
The induction requires a fixed degree bound, so we state the main result in that form.

\begin{theorem}\label{thm:main}
For every integer $D\ge 6$, every finite simple planar graph $G$ with $\Delta(G)\le D$ has a
$D$-coloring with at most $2D-1$ colors.
\end{theorem}

The letter $D$ in ``$D$-coloring'' refers to diamonds; the integer $D$ in
Theorem~\ref{thm:main} is a degree bound, fixed throughout the proof, and the palette is
$\Pal=\{1,\dots,2D-1\}$.

\begin{corollary}[Wang's conjecture]\label{cor:wang}
Every finite simple planar graph $G$ satisfies
\[
  \chiD(G)\le
  \begin{cases}
    9, & \Delta(G)\le 4,\\
    10, & \Delta(G)=5,\\
    2\Delta(G)-1, & \Delta(G)\ge 6.
  \end{cases}
\]
\end{corollary}

\begin{proof}
The first two bounds are proved in \cite{HKW}.  For the last, apply
Theorem~\ref{thm:main} with $D=\Delta(G)$.
\end{proof}

As Wang observed \cite{Wang}, the bound $2D-1$ cannot be lowered.  Let $\Bk_{D-1}$ be the
\emph{book} consisting of an edge $xy$ together with $D-1$ common neighbors $z_1,\dots,z_{D-1}$
of $x$ and $y$.  It is planar and has maximum degree $D$ and $2D-1$ edges.  Any two of its edges
are incident, or are of the form $xz_i,yz_j$ with $i\ne j$ and lie in the diamond with spine $xy$
and tips $z_i,z_j$.  Hence all $2D-1$ edges need distinct colors (see also
Remark~\ref{rem:books}).

\subsection*{Overview}

We take a counterexample with fewest vertices, then fewest edges
(Section~\ref{sec:mincounter}).  It is $3$-connected, has minimum degree at least four,
contains no $\Bk_{D-1}$, and has every edge in a triangle.

If $G$ has separating triangles, we take an innermost one and work only inside it, where every
triangle with an interior vertex bounds a face.  By Euler's formula the charges inside sum to
minus the number of edges from the triangle to its interior (Section~\ref{sec:disk}).  For
background on discharging see \cite{CranstonWest}.

The lists left after deleting one vertex of degree four or five settle $D\ge 9$ directly, and
$D=8$ with a rule moving charge $\tfrac14$ from vertices of degree eight to adjacent vertices of
degree five (Section~\ref{sec:star}).  For $D=7$ we use three more reducible configurations and
a rule moving charge $\tfrac13$ from vertices of degree seven to adjacent vertices of degree four
or five (Section~\ref{sec:d7}).

For $D=6$ we delete a facial triangle containing a vertex of degree at most five.  The possible
neighborhoods of such a triangle lead to $1{,}292$ list-coloring problems, which we check by
computer (Section~\ref{sec:d6} and Appendix~\ref{app:comp}).  This is the only use of a
computer in the proof.

\section{Preliminaries}\label{sec:prelim}

Graphs are finite and simple.  A plane graph is a planar graph with a fixed embedding.  For a
vertex $v$ we write $N(v)$ for its neighborhood and $d(v)=|N(v)|$; a vertex of degree four is a
\emph{four-vertex}, and similarly for other degrees.  For $S\subseteq V(G)$, $G-S$ is the
subgraph induced by $V(G)\setminus S$.  The \emph{link} of $v$ is $G[N(v)]$; in a
plane graph the neighbors of $v$ carry the cyclic order of the rotation at $v$, and the
\emph{rotation cycle} at $v$ joins cyclically consecutive neighbors.  All common-neighbor and
conflict counts refer to the original graph, even when some of its edges are temporarily
uncolored.

\subsection{Conflicts and partial colorings}

Two distinct edges \emph{conflict} (or \emph{see} each other) if they are incident or lie in a
common diamond.  The \emph{conflict graph} has vertex set $E(G)$ and an edge for each
conflicting pair, so a $D$-coloring is exactly a proper vertex coloring of the conflict graph.
A nonincident edge that conflicts with $e$ is a \emph{partner} of $e$.

\begin{lemma}\label{lem:partner}
The partners of an edge $uv$ are exactly the edges $tq$ with
\[
  t\in N(u)\cap N(v),\qquad q\in (N(u)\cup N(v))\setminus\{u,v,t\},\qquad tq\in E(G).
\]
Hence $uv$ sees $d(u)+d(v)-2$ incident edges together with its distinct partners.  Two disjoint
edges $ab,cd$ conflict if and only if at least three of $ac,ad,bc,bd$ are edges.
\end{lemma}

\begin{proof}
In a diamond with spine $st$ and tips $p,q$, the nonincident pairs are exactly $\{sp,tq\}$ and
$\{sq,tp\}$.  By symmetry we may assume $\{u,v\}=\{s,p\}$; then the other member $tq$ has an
endpoint $t$ adjacent to both $u$ and $v$, and its other endpoint $q$ is adjacent to
$s\in\{u,v\}$.  Conversely, for such $t$ and $q$, the edges $uv,ut,vt,tq$ together with one of
$uq,vq$ form a diamond in which $uv$ and $tq$ are nonincident.  The $d(u)-1$ other edges at $u$ and the $d(v)-1$ other edges at $v$ are
distinct.  Finally, a diamond containing two disjoint edges has exactly their four endpoints and
contains at least three of the four cross edges; conversely, two disjoint edges together with
three cross edges span five edges on four vertices, which form a diamond containing both
edges.
\end{proof}

\begin{remark}\label{rem:books}
The \emph{book} at an edge $st$ consists of $st$ and all edges from $s$ or $t$ to common
neighbors of $s$ and $t$.  A proper edge coloring is a $D$-coloring if and only if every book is
rainbow: a diamond with spine $st$ lies in the book at $st$, and two
edges of a book are incident or lie in a common diamond with spine $st$.  A book at an edge whose
endpoints have degree at most $D$ has at most $2D-1$ edges, with equality exactly when $s$ and $t$
have $D-1$ common neighbors, that is, when the book is a copy of $\Bk_{D-1}$.
\end{remark}

\begin{lemma}\label{lem:safe}
For $S\subseteq V(G)$, a $D$-coloring of $G-S$ is a valid partial $D$-coloring of $G$: no two
conflicting edges of $G-S$ have the same color.  Uncoloring further edges preserves this
property, and the restriction of a $D$-coloring of $G$ to a subgraph is a $D$-coloring of that
subgraph.
\end{lemma}

\begin{proof}
Incidences between edges of $G-S$ are the same in $G$.  If two disjoint edges of $G-S$ lie in a
diamond of $G$, that diamond has exactly their four endpoints, none of which is in $S$; so the
diamond is a subgraph of the induced graph $G-S$, and the constraint was already enforced.
Uncoloring cannot create a monochromatic conflicting pair.  Every incidence or diamond of a
subgraph is also one of $G$.
\end{proof}

Deleting a single edge is different: if it is the spine of a diamond whose tips are nonadjacent,
the two nonincident pairs of that diamond survive but need no longer conflict.  We delete an edge
alone only after showing that it lies in no triangle, hence in no diamond
(Lemma~\ref{lem:triangle}).

\begin{lemma}\label{lem:lists}
Let $S\subseteq V(G)$, start from a $D$-coloring of $G-S$, uncolor some further edges, and let
$U$ be the set of uncolored edges of $G$.  For $e\in U$ let $L(e)$ be the set of colors of $\Pal$ not used on any
colored edge that sees $e$.  If at most $I(e)$ colored incident edges and at most $P(e)$
colored partners see $e$, then $|L(e)|\ge |\Pal|-I(e)-P(e)$.  Every proper coloring of the
conflict graph on $U$ from the lists $L$ completes the partial coloring to a $D$-coloring.
\end{lemma}

\begin{proof}
Each colored edge seeing $e$ removes at most one color.  Conflicts between colored and
uncolored edges are respected by the definition of $L$, conflicts inside $U$ by properness, and
conflicts among colored edges by Lemma~\ref{lem:safe}.
\end{proof}

\subsection{List-coloring tools}

\begin{lemma}[Hall \cite{Hall}]\label{lem:hall}
Finite lists $L_1,\dots,L_n$ have distinct representatives if and only if
$|\bigcup_{i\in J}L_i|\ge |J|$ for every $J\subseteq\{1,\dots,n\}$.  In particular, a clique whose
list sizes, in increasing order, are $g_1\le\dots\le g_n$ with $g_i\ge i$ for all $i$ is
list-colorable: color it in this order.
\end{lemma}

\begin{lemma}\label{lem:evencycle}
Every even cycle is colorable from lists of size at least two.
\end{lemma}

\begin{proof}
Shrink every list to two colors.  If all lists are equal, alternate the two colors.  Otherwise
some adjacent pair $u,w$ has different lists; delete the edge $uw$ to obtain a path from $u$ to
$w$, give $u$ a color not in $L(w)$, and color the path greedily from $u$ to $w$.  At $w$ only
the color of its path predecessor needs to be avoided.
\end{proof}

The \emph{graph polynomial} of a graph $H$ on the ordered vertex set $\{0,\dots,m-1\}$ is
\[
  P_H=\prod_{ij\in E(H),\,i<j}(X_i-X_j).
\]

\begin{lemma}[\cite{Alon,AlonTarsi}]\label{lem:cn}
Let $t_0,\dots,t_{m-1}\ge0$ be integers with $\sum_i t_i=|E(H)|$.  If the coefficient of
$X_0^{t_0}\cdots X_{m-1}^{t_{m-1}}$ in $P_H$ is nonzero, then $H$ is colorable from every list
assignment with $|L_i|\ge t_i+1$ for all $i$.
\end{lemma}

\begin{lemma}\label{lem:greedy}
Let $f\colon V(H)\to\mathbb Z$.  If $V(H)$ has an order $v_1,\dots,v_m$ (a \emph{removal order})
such that each $v_i$ has fewer than $f(v_i)$ neighbors among $v_{i+1},\dots,v_m$, then $H$ is
colorable from all lists with $|L(v)|\ge f(v)$: color $v_m,\dots,v_1$ in this order.  If $H$ is
colorable from all lists of sizes at least $f$, then every subgraph $H'$ of $H$ is colorable
from all lists whose sizes are at least the corresponding values of $f$.
\end{lemma}

\begin{proof}
When $v_i$ is colored, its colored neighbors are among $v_{i+1},\dots,v_m$, so a color of its
list remains.  For $H'$, give each vertex of $H$ missing from $H'$ an arbitrary list of the
required size, shrink the other lists to the required sizes, color $H$, and restrict; the
additional edges of $H$ only impose additional constraints.
\end{proof}

\subsection{Planar facts}

We use Euler's formula $n-e+f=2$ for connected plane graphs, the bound $e\le 3n-6$ for simple
planar graphs with $n\ge3$, the bound $e\le 2n-4$ for simple bipartite planar graphs with
$n\ge3$ (so $K_{3,3}$ is not planar), the Jordan curve theorem (an edge cannot join two vertices
lying strictly on opposite sides of a cycle), and the fact that minors of planar graphs are
planar.  In a $2$-connected plane graph every face is bounded by a cycle
\cite[Proposition~4.2.6]{Diestel}.

\begin{lemma}\label{lem:join}
If the join $K_2+J$ is planar, then $J$ is a linear forest, that is, a disjoint union of paths.
\end{lemma}

\begin{proof}
Let $x,y$ be the vertices of $K_2$.  A vertex $w$ of $J$ with three neighbors $z_1,z_2,z_3$ in
$J$ gives a $K_{3,3}$ with parts $\{x,y,w\}$ and $\{z_1,z_2,z_3\}$.  A cycle of length $r\ge3$ in
$J$ together with $x,y$ spans $r+2$ vertices and at least $3r+1>3(r+2)-6$ edges.  So $J$ has
maximum degree at most two and no cycle.
\end{proof}

\section{The minimal counterexample}\label{sec:mincounter}

Fix an integer $D\ge 6$, put $k=2D-1$, and let $\Pal=\{1,\dots,k\}$.  Suppose that some planar
graph of maximum degree at most $D$ has no $D$-coloring with palette $\Pal$, and choose such a
graph $G$ with the fewest vertices and, subject to this, the fewest edges.  We fix a plane
embedding of $G$ whenever an argument uses regions.  For every nonempty $S\subseteq V(G)$ the
graph $G-S$ has a $D$-coloring with palette $\Pal$, which by Lemma~\ref{lem:safe} is a valid
partial coloring of $G$; extension lists are always those of Lemma~\ref{lem:lists}.

\subsection{Basic structure}

\begin{lemma}\label{lem:mindeg3}
$G$ has minimum degree at least three.
\end{lemma}

\begin{proof}
An isolated vertex can be deleted and restored.  If $N(v)=\{a\}$, delete $v$: the edge $va$ has
no partner and sees at most $D-1$ colored edges at $a$, so at least $D$ colors remain.  Let
$N(v)=\{a,b\}$ and delete $v$.  If $ab\notin E(G)$, the spokes $va,vb$ have no partners and
lists of size at least $D$.  If $ab\in E(G)$, the partners of $va$ are, by
Lemma~\ref{lem:partner}, the edges $bq$ with $q\in N(a)\cap N(b)\setminus\{v\}$; there are at most
$D-2$ of them, because $a$ and $v$ are two neighbors of $b$.  As $va$ also sees at most $D-1$
colored edges at $a$, its list has size at least $k-(D-1)-(D-2)=2$, and symmetrically for $vb$.
Two incident edges with lists of size at least two receive distinct colors.  This contradicts
minimality.
\end{proof}

\begin{lemma}\label{lem:3conn}
$G$ is $3$-connected.
\end{lemma}

\begin{proof}
Edges in different components never conflict, so $G$ is connected.

Suppose $z$ is a cut vertex.  Color each subgraph induced by $z$ and one component of $G-z$.  A
diamond is $2$-connected, so no diamond meets two components.  In each piece the edges at $z$
have distinct colors; permute the palette of each piece so that its colors at $z$ form a
prescribed set, these sets being pairwise disjoint, which is possible because their total size
is $d(z)\le D\le k$.  Properness at $z$ is then the only constraint between pieces, and it
holds.

Suppose $\{a,b\}$ is a separating pair and let $C_1,\dots,C_m$ ($m\ge2$) be the components of
$G-\{a,b\}$.  Each $C_i$ has neighbors at both $a$ and $b$, as otherwise $G$ would have a cut
vertex.  Color each $G_i=G[C_i\cup\{a,b\}]$ by minimality.  Call the edges of $G_i$ at $a$ or
$b$ its \emph{boundary edges}, and let $s_i$ be the number of color classes of $G_i$ meeting
its boundary edges.

First let $ab\in E(G)$.  Every boundary edge is incident with $ab$, so the color of $ab$ occurs on
no other boundary edge.  Prescribe one common color for $ab$ and pairwise disjoint sets of
$s_i-1$ further colors for the other boundary classes of the pieces.  The number of colors needed
is at most
\[
  1+\sum_{i=1}^m(s_i-1)\le 1+\sum_{i=1}^m\bigl(d_{G_i}(a)+d_{G_i}(b)-2\bigr)=d_G(a)+d_G(b)-1\le k ,
\]
and each prescribed injection extends to a permutation of $\Pal$.  A diamond meeting two
components contains $a$, $b$ and one vertex of each of two components, since it has no cut
vertex and only four vertices.  These two vertices are nonadjacent, so they are its tips and
$ab$ is its spine; all five edges are boundary edges, and the prescribed palettes make their
colors distinct.

Now let $ab\notin E(G)$.  A set of four vertices consisting of $a$, $b$ and one vertex in each of two
components misses both $ab$ and the edge between the components, so no diamond meets two
components and only properness at $a$ and $b$ must be enforced across pieces.  There are at
most $d(a)+d(b)\le 2D$ boundary edges, and $\sum_i s_i$ is no larger.  If $\sum_i s_i\le k$,
allocate disjoint boundary palettes.  Otherwise $\sum_i s_i=2D$, there are exactly $2D$
boundary edges, and every boundary color class of every piece contains exactly one boundary
edge.  Choose an edge at $a$ in $G_1$ and an edge at $b$ in $G_2$, give their two classes the same
color, and make all other boundary classes distinct; this uses $2D-1$ colors.  The two chosen
edges are nonincident and lie in no common diamond, and the prescribed map is injective on the
boundary colors of each piece, so it extends to a palette permutation of that piece.

Each case contradicts minimality.  Since $G$ has a vertex of degree at least three
(Lemma~\ref{lem:mindeg3}), $|V(G)|\ge4$, and $G$ is $3$-connected.
\end{proof}

\begin{lemma}\label{lem:nobook}
$G$ contains no book $\Bk_{D-1}$.  Consequently every edge of $G$ has at most $D-2$ common
neighbors.
\end{lemma}

\begin{proof}
Suppose the edge $xy$ has a set $Z$ of $D-1$ common neighbors.  Then $N(x)=Z\cup\{y\}$ and
$N(y)=Z\cup\{x\}$.  A component of $G-(\{x,y\}\cup Z)$ can attach only to $Z$.  With at most two
attachments it contradicts Lemma~\ref{lem:3conn}; with at least three attachments, among them
$z_1,z_2,z_3$, contracting it to a vertex gives a $K_{3,3}$ minor with parts
$\{x,y,\text{the component}\}$ and $\{z_1,z_2,z_3\}$.  Hence $V(G)=\{x,y\}\cup Z$.  By Lemma~\ref{lem:join}, $G[Z]$ is a linear
forest, and it is connected because $G-\{x,y\}$ is connected.  So $G[Z]$ is a path
$z_1z_2\cdots z_m$ with $m=D-1\ge5$, and $G=K_2+P_m$.  Define
\[
  \varphi(xy)=2m+1,\qquad \varphi(xz_j)=j,\qquad \varphi(yz_j)=m+j\quad(1\le j\le m),
\]
\[
  \varphi(z_iz_{i+1})=
  \begin{cases} i+3, & 1\le i\le m-3,\\ i-2, & i\in\{m-2,m-1\}.\end{cases}
\]
The $2m+1$ edges $xy,xz_j,yz_j$ receive distinct colors.  By Lemma~\ref{lem:partner}, the rim
edge $z_iz_{i+1}$ sees exactly $xy$, the edges $xz_j,yz_j$ for $j\in\{i-1,i,i+1,i+2\}$ (when these
indices exist), and its adjacent rim edges.  Its color lies in $\{1,\dots,m\}\setminus\{i-1,\dots,i+2\}$
(here $m\ge5$ is used for $i=m-2$), so it differs from the colors of all spokes and of $xy$ that
it sees.  Consecutive rim colors differ: they increase by one up to $\varphi(z_{m-3}z_{m-2})=m$,
followed by $m-4$ and $m-3$.  Nonincident rim edges do not conflict, since at most one cross edge
between them lies on the path.  Thus $\varphi$ is a $D$-coloring with $k$ colors, a
contradiction.
\end{proof}

\begin{lemma}\label{lem:triangle}
Every edge of $G$ lies in a triangle.
\end{lemma}

\begin{proof}
Every edge of a diamond lies in a triangle of that diamond, so an edge $e$ in no triangle lies in
no diamond, and deleting $e$ removes no constraint between other edges.  By edge-minimality
$G-e$ has a $D$-coloring; $e$ sees at most $2D-2$ incident edges and no partners, so one of the
$k$ colors is available for $e$, a contradiction.
\end{proof}

\subsection{Vertices of degree three}

\begin{lemma}\label{lem:deg3easy}
Let $v$ have degree three, with $N(v)=\{a,b,c\}$.  At most one vertex other than $v$ is adjacent
to all of $a,b,c$.  A coloring of $G-v$ extends to $G$ if $G[\{a,b,c\}]$ is not a triangle, or if
it is a triangle and no vertex other than $v$ is adjacent to all of $a,b,c$.
\end{lemma}

\begin{proof}
Two vertices other than $v$ adjacent to all of $a,b,c$ would, together with $v$, form a $K_{3,3}$
against $a,b,c$; this proves the first assertion.  The three spokes form a clique, so we need
distinct representatives of their lists (Lemma~\ref{lem:hall}).  If $G[\{a,b,c\}]$ has no edge,
each list has size at least $D$.  If it
has exactly one edge, say $ab$, put $h=|N(a)\cap N(b)\setminus\{v\}|\le D-2$.  The numbers of
colored edges seeing $va,vb,vc$ are at most $d(a)-1+h$, $d(b)-1+h$ and $d(c)-1$, so the lists
have sizes at least $2,2,D$, and Hall's condition holds.

If $G[\{a,b,c\}]$ is a path $a-b-c$, put $A=N(a)\cap N(b)\setminus\{v\}$ and
$C=N(c)\cap N(b)\setminus\{v\}$.  Both lie in $N(b)\setminus\{v,a,c\}$, which has at most $D-3$
vertices, and $|A\cap C|\le1$, since two common members together with $v$ would form a $K_{3,3}$
against $a,b,c$.  Hence $|A|+|C|\le D-2$.  The exact numbers of colored edges seeing $va,vb,vc$
are
\[
  d(a)+|A|,\qquad d(b)-1+|A|+|C|,\qquad d(c)+|C| .
\]
Each is at most $2D-3$, and one of $|A|,|C|$ is at most $\lfloor(D-2)/2\rfloor\le D-4$ (as $D\ge6$),
so the corresponding end spoke has at least three colors and all three lists have at least two.
Hall's condition holds.

Finally let $abc$ be a triangle, put $m=D-3$, and let
\[
  X_a=N(a)\setminus\{v,b,c\},\quad X_b=N(b)\setminus\{v,a,c\},\quad X_c=N(c)\setminus\{v,a,b\},
\]
$k_a=|X_a|$, $k_b=|X_b|$, $k_c=|X_c|$ (each at most $m$), and $h_{ij}=|X_i\cap X_j|$.  By
Lemma~\ref{lem:partner}, the partners of $va$ are $bc$, the edges $bx$ with $x\in X_a\cap X_b$, and
the edges $cx$ with $x\in X_a\cap X_c$; symmetrically for $vb$ and $vc$.  Hence the lists have
sizes at least
\begin{equation}\label{eq:f3}
  f_a=2m+2-k_a-h_{ab}-h_{ac},\quad f_b=2m+2-k_b-h_{ab}-h_{bc},\quad f_c=2m+2-k_c-h_{ac}-h_{bc}.
\end{equation}
The triple intersection $X_a\cap X_b\cap X_c$ consists of the vertices other than $v$ adjacent to
all of $a,b,c$, so by hypothesis it is empty.  Then the two pair
intersections at each vertex are disjoint, each incident $h$-sum is at most the corresponding
$k_i$, and $f_a,f_b,f_c\ge2$.  If one bound is at least three, Hall's condition holds.
Otherwise $k_a=k_b=k_c=m$ and every incident $h$-sum equals $m$, so
$h_{ab}=h_{ac}=h_{bc}=m/2$; this is impossible for odd $m$.  For even $m$, let $P_{ij}$ be the set
of the two edges from $i$ and $j$ to each vertex of $X_i\cap X_j$, so $|P_{ij}|=m$.  The sets of
colored edges seeing the three spokes are exactly
\[
  E_a=\{ab,ac,bc\}\cup P_{ab}\cup P_{ac},\quad E_b=\{ab,ac,bc\}\cup P_{ab}\cup P_{bc},\quad
  E_c=\{ab,ac,bc\}\cup P_{ac}\cup P_{bc},
\]
each of size $k-2$.  If Hall's condition failed, all three lists would be one common two-set,
so each $E_i$ would be rainbow.  Every two of the groups $P_{ab},P_{ac},P_{bc}$ lie in a common
$E_i$, so their $3m$ edges would have distinct colors, avoiding the three colors of the triangle
and the two list colors; but only $k-5=2m$ colors remain.
\end{proof}

\begin{lemma}\label{lem:deg3}
$G$ has no vertex of degree three.
\end{lemma}

\begin{proof}
Let $N(v)=\{a,b,c\}$.  By Lemma~\ref{lem:deg3easy} we may assume that $abc$ is a triangle and
that a unique vertex $w\ne v$ is adjacent to $a,b,c$.  Delete $v$ and fix a coloring $\varphi$ of
$G-v$; the six edges $K=E(K_4[a,b,c,w])$ have distinct colors.  Put $n=D-4\ge2$.  Let $r,s,t$ be
the numbers of common neighbors of $ab,ac,bc$ outside $\{v,c,w\}$, $\{v,b,w\}$, $\{v,a,w\}$,
respectively; these $r+s+t$ vertices are distinct because $w$ is the only further common neighbor of
$a,b,c$.  Let $A,B,C$ be the numbers of remaining neighbors of $a,b,c$; we call these neighbors
\emph{private}, and each of them is adjacent to only one of $a,b,c$ (Figure~\ref{fig:deg3}).
Then
\begin{equation}\label{eq:budget3}
  r+s+A\le n,\qquad r+t+B\le n,\qquad s+t+C\le n,\qquad 0\le r,s,t\le n-1,
\end{equation}
the last because, for example, $r\ge n$ would give $ab$ at least $D-1$ common neighbors,
contrary to Lemma~\ref{lem:nobook}.  Let $P_{ab}$ be the set of the $2r$ edges from $a$ and $b$
to the additional common neighbors of $ab$, and define $P_{ac},P_{bc}$ similarly.  Let $I_a$ be
the set of the $A$ edges from $a$ to its private neighbors, and define $I_b,I_c$ similarly.  The
sets of colored edges seeing the spokes, and the resulting list bounds, are
\[
\begin{aligned}
  E_a&=K\cup P_{ab}\cup P_{ac}\cup I_a, & f_a&=2n+1-2(r+s)-A,\\
  E_b&=K\cup P_{ab}\cup P_{bc}\cup I_b, & f_b&=2n+1-2(r+t)-B,\\
  E_c&=K\cup P_{ac}\cup P_{bc}\cup I_c, & f_c&=2n+1-2(s+t)-C.
\end{aligned}
\]
Writing $g=n-r-s\ge0$, we have $0\le A\le g$ and $f_a=1+2g-A$, so $f_a\ge1$, and
\begin{equation}\label{eq:f12}
  f_a=1\iff r+s=n\text{ and }A=0,\qquad f_a=2\iff r+s=n-1\text{ and }A=1;
\end{equation}
likewise for $f_b$ with $(r+t,B)$ and for $f_c$ with $(s+t,C)$.

\begin{figure}[ht]
\centering
\includegraphics{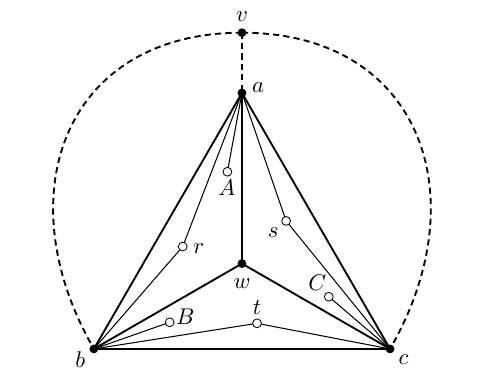}
\caption{The configuration of Lemma~\ref{lem:deg3}, with the face $abc$ of $K_4[a,b,c,w]$ drawn as
the outer face.  The dashed edges are the spokes at $v$.  Open circles labeled $r,s,t$ stand for the
additional common neighbors of $ab,ac,bc$, and those labeled $A,B,C$ for the private neighbors of
$a,b,c$; one of each is drawn, although each of these classes may be empty, and edges not used
in the argument are omitted.}
\label{fig:deg3}
\end{figure}

The subgraph $K_4[a,b,c,w]$ has four triangular regions, and $v$ lies in the region $abc$.  It is
the only vertex there: a vertex in one of the regions $abv,bcv,cav$ could reach the rest of $G$
only through two of $a,b,c$ (as $d(v)=3$), contradicting $3$-connectivity.  So every additional common neighbor of
$ab$ lies in the region $abw$, and similarly for $ac$ and $bc$.

\emph{A direct extension.}  If some neighbor of $v$, say $c$, has degree at most four, then
$N(c)=\{v,a,b,w\}$, so $s=t=C=0$, and \eqref{eq:budget3} gives $f_a,f_b\ge n+1-r\ge2$ and
$f_c=2n+1\ge5$; Hall's condition holds.  With Lemma~\ref{lem:deg3easy} this shows: \emph{if a
vertex $u$ has degree three and a neighbor of degree at most four, then every $D$-coloring of
$G-u$ extends to $G$.}

By Lemma~\ref{lem:hall}, Hall's condition holds unless at least two of $f_a,f_b,f_c$ equal one, or
the bounds are $2,2,2$, or they are $1,2,2$ in some order.

\emph{At least two bounds equal one.}  Say $f_a=f_b=1$.  By \eqref{eq:f12}, $r+s=r+t=n$ and
$A=B=0$, so $s=t$, $2s+C\le n$, and $r\le n-1$ forces $s\ge1$; hence $r=n-s\ge\lceil n/2\rceil$.
As $A=B=0$, the only neighbors of $a$ or $b$ strictly inside $abw$ are the $r$ additional common
neighbors of $ab$.  Their paths through $a$ and $b$ are nested in the disk $abw$; order them
$p_1,\dots,p_r$ from $w$ towards $ab$, and put $p_0=w$.  The fans at $a$ and $b$ cut the disk into
the regions $ap_{i-1}bp_i$ ($1\le i\le r$) and the triangle $abp_r$.  A further vertex in
$ap_{i-1}bp_i$ could attach only to $p_{i-1},p_i$, and one in $abp_r$ only to $p_r$, contradicting
$3$-connectivity.  An edge between nonconsecutive vertices of $p_0,\dots,p_r$ would cross one of
the nested paths, and since $G-\{a,b\}$ is connected, every edge $p_{i-1}p_i$ is present.  If
$r\ge2$, then $N(p_r)=\{a,b,p_{r-1}\}$ and $N(p_{r-1})\subseteq\{a,b,p_r,p_{r-2}\}$, so by the
direct extension every coloring of $G-p_r$ extends to $G$, contradicting minimality.

If $r=1$, then $n=2$, $s=t=1$ and $A=B=C=0$.  By the argument just given, applied also to $ac$ and
$bc$, each of the regions $abw,acw,bcw$ contains exactly one vertex, adjacent to its three
corners.  So $G$ is $K_4$ with one new vertex in each face, joined to the corners of that face.
For $j\in\{a,b,c,w\}$ let $\bar\jmath$ be the new vertex in the face opposite $j$ (so $\bar w=v$),
and color $ab,ac,aw,bc,bw,cw$ with $1,\dots,6$.  Each of the twelve arcs $i\to j$ between
$a,b,c,w$ lies in exactly one of the directed triangles
\[
  (a,c,w),\qquad (a,b,c),\qquad (a,w,b),\qquad (b,w,c),
\]
and we give the edge $i\bar\jmath$ the color $7$, $8$, $9$ or $10$ according to the directed
triangle containing $i\to j$, in the order listed.  Two arcs of one directed triangle can be
written $i\to j$ and $j\to l$; the edges $i\bar\jmath$ and $j\bar l$ are disjoint, and of their
cross edges only $ij$ and $i\bar l$ are present, so they do not conflict
(Lemma~\ref{lem:partner}).  This $D$-coloring of $G$ with ten colors is a contradiction.

\emph{Bounds $2,2,2$.}  By \eqref{eq:f12}, $A=B=C=1$ and $r=s=t=(n-1)/2$, so $n$ is odd and
$n\ge3$.  Hall's condition fails only if the three lists are one common two-set, and then each
$E_i$, of size $k-2$, is rainbow.  Each group $P_{ij}$ has $n-1$ edges and every two groups lie in
a common $E_i$, so their $3(n-1)$ colors are distinct and avoid the six colors of $K$ and the two
list colors.  This requires $3(n-1)\le k-8=2n-1$, contrary to $n\ge3$.

\emph{Bounds $1,2,2$.}  Say $f_a=1$ and $f_b=f_c=2$.  By \eqref{eq:f12}, $A=0$, $B=C=1$,
$r=s=n/2$ and $t=n/2-1$.  If $L(va)=\{\alpha\}$, all $k-1$ edges of $E_a=K\cup P_{ab}\cup P_{ac}$
have distinct colors, so
\[
  L(vb)\subseteq\{\alpha\}\cup\varphi(P_{ac}),\qquad L(vc)\subseteq\{\alpha\}\cup\varphi(P_{ab}),
\]
where $\varphi(P_{ab})$ and $\varphi(P_{ac})$ are disjoint and avoid $\alpha$; these lists meet at
most in $\alpha$ and have at least two colors each, and Hall's condition holds.  Otherwise Hall's
condition can fail only if the three lists are one common two-set $\{\alpha,\beta\}$, and then
$E_b$ and $E_c$, of size $k-2$, are rainbow.  If $n\ge4$, then $P_{bc}\ne\emptyset$, and a color on
$P_{bc}$ is absent from $K\cup P_{ab}$ (by $E_b$) and from $P_{ac}$ (by $E_c$), hence from $E_a$; it
lies in $L(va)$ but not in $L(vb)$, a contradiction.

It remains that $n=2$ (so $D=6$), $r=s=1$, $t=0$, and
\begin{equation}\label{eq:spare}
  L(va)=L(vb)=L(vc)=\{\alpha,\beta\}.
\end{equation}
Let $p,q$ be the additional common neighbors of $ab$ and $ac$, and $b',c'$ the private neighbors
of $b$ and $c$.  These are distinct from each other and from $v,a,b,c,w$ ($p\ne q$ because $w$ is
the only further common neighbor of $a,b,c$, and $b'\ne c'$ because $t=0$), and
\[
  N(a)=\{v,b,c,w,p,q\},\qquad N(b)=\{v,a,c,w,p,b'\},\qquad N(c)=\{v,a,b,w,q,c'\},
\]
\begin{equation}\label{eq:E122}
  E_a=K\cup\{ap,bp,aq,cq\},\qquad E_b=K\cup\{ap,bp,bb'\},\qquad E_c=K\cup\{aq,cq,cc'\}.
\end{equation}
By \eqref{eq:spare}, $E_b$ and $E_c$ are rainbow, and no colored edge at $a$, $b$ or $c$ has a
\emph{spare} color $\alpha$ or $\beta$.

The vertices $p$ and $q$ lie in the regions $abw$ and $acw$.  By $3$-connectivity, every
component of $G-\{a,b,c,w\}$ inside one of these regions is adjacent to all three corners of
that region, so $w$ has a neighbor in each of the two regions.  As $w$ has at most three
neighbors outside $\{a,b,c\}$, it has exactly one neighbor in at least one of the two regions.
By the symmetry $b\leftrightarrow c$, $p\leftrightarrow q$, $b'\leftrightarrow c'$ we may assume
that this region is $abw$, and we call this neighbor $z$.  Let $R$ be the set of vertices
strictly inside $abw$.  The only edge from $a$ into $R$ is $ap$, and the edges from $b$ into $R$
are $bp$ and possibly $bb'$.  A diamond with a vertex in $R$ and a vertex strictly outside the
closed region $abw$ has exactly two vertices in $\{a,b,w\}$, since it is $2$-connected; these form
its spine, and its tips lie on opposite sides.  So, apart from $ab,aw,bw$, only edges from
$\{a,b,w\}$ into $R$ can conflict with edges outside the closed region.

We shall give one edge $e\in\{ap,bp\}$ a spare color $\gamma$.  Write $c_0=\varphi(e)$; as $e$ lies
in $E_b$, the color $c_0$ is not spare.  We choose $e$ and $\gamma$ so that
\begin{enumerate}[label=\textup{(\alph*)}]
  \item $\varphi(wz)\notin\{c_0,\gamma\}$, and
  \item if $z=p$, then $\varphi(wq)\ne\gamma$ in case $e=ap$, and $\varphi(wb')\ne\gamma$ in case
        $e=bp$; a condition on an edge that is absent from $G$ holds automatically.
\end{enumerate}
If $z\ne p$, choose $e$ with $\varphi(e)\ne\varphi(wz)$, which is possible as
$\varphi(ap)\ne\varphi(bp)$, and a spare color $\gamma\ne\varphi(wz)$.  If $z=p$, then
$c_0\ne\varphi(wp)$ for both choices of $e$, since $wp$ meets $ap$ and $bp$.  If moreover
$\varphi(wp)$ is not spare, take $e=ap$ and a spare color $\gamma\ne\varphi(wq)$.  If
$\varphi(wp)$ is spare, let $\gamma$ be the other spare color, and take $e=ap$ unless $wq$ is an
edge with $\varphi(wq)=\gamma$; in that case take $e=bp$, and then $\varphi(wb')\ne\gamma$
because $wb'$ and $wq$ meet at $w$.  In every case (a) and (b) hold.

Now exchange the colors $c_0$ and $\gamma$ on all edges of $G[R\cup\{a,b,w\}]$.  Among the edges
of this subgraph at $a$, $b$ or $w$, only $e$ has color $c_0$ or $\gamma$: the edges $ab,aw,bw$,
the other edge of $\{ap,bp\}$ and the edge $bb'$ lie in the rainbow set $E_b$ together with $e$,
and no edge of $E_b$ has a spare color, while the color of $wz$ is neither by (a).  So $e$ now has
color $\gamma$, and every other edge from $\{a,b,w\}$ into $R$ keeps its color.  Conflicts between
two edges of the subgraph are respected, since we permuted colors there.  As shown above, apart
from $ab,aw,bw$ only edges from $\{a,b,w\}$ into $R$ can conflict with edges outside the closed
region; so it remains to compare $e$ with the conflicting edges $e'$ that have an end strictly
outside the closed region.  If $e'$ meets $e$, it does so at $a$ or $b$, and no other colored
edge at $a$ or $b$ has the spare color $\gamma$.  Otherwise $e$ and $e'$ lie in a diamond whose spine is $ab$, $aw$ or $bw$.  If the spine
is $ab$, the tip outside is a common neighbor of $a$ and $b$, namely $v$ or $c$, so $e'$ is $av$ or
$bv$, which are uncolored, or $ac$ or $bc$, whose colors lie in $\varphi(K)$ and so are not
$\gamma$.  If the spine is $aw$ or $bw$, then $p$ is a tip, so $wp\in E(G)$ and $z=p$, and $e'$
joins $w$ to a common neighbor of $w$ and the end of $e$ in $\{a,b\}$.  So $e'$ is $wc$, whose
color lies in $\varphi(K)$, or $wq$ if $e=ap$, or $wb'$ if $e=bp$, and by (b) its color is not
$\gamma$.  Hence the recolored $\varphi$ is a valid coloring of $G-v$.

Of the edges in \eqref{eq:E122} only $e$ has changed.  It is not in $E_c$, and in $E_b$ the color
$c_0$ occurred only on $e$.  Let $\delta$ be the spare color other than $\gamma$, and put
$\varphi(va)=\delta$, $\varphi(vb)=c_0$ and $\varphi(vc)=\gamma$ (exchanging $vb$ and $vc$ if the
symmetry was used).  These three colors are distinct and are absent from $E_a$, $E_b$ and $E_c$,
respectively, so $\varphi$ extends to $G$, a contradiction.
\end{proof}

\begin{corollary}\label{cor:structure}
For every $D\ge6$ a minimal counterexample is $3$-connected, has minimum degree at least four,
contains no $\Bk_{D-1}$, and has every edge in a triangle.
\end{corollary}

\begin{proof}
Combine Lemmas~\ref{lem:mindeg3}--\ref{lem:triangle} and~\ref{lem:deg3}.
\end{proof}

\section{Innermost separating triangles}\label{sec:disk}

Fix $D\ge6$ and a minimal counterexample $G$ with a plane embedding; it has the properties
listed in Corollary~\ref{cor:structure}.  A triangle is \emph{separating} if both of its sides
contain vertices.

\begin{lemma}\label{lem:innermost}
Suppose $G$ has a separating triangle, and choose one, $T=xyz$, with as few vertices inside as
possible.  Let $I\ne\emptyset$ be the set of vertices strictly inside $T$.  Then:
\begin{enumerate}
  \item every triangle with a vertex in $I$ bounds a face, and every edge with an endpoint in
        $I$ lies in at most two triangles;
  \item for $v\in I$ the link of $v$ is a subgraph of its rotation cycle; in particular it has
        no chord;
  \item each edge of $T$ has at most one common neighbor in $I$, and no vertex of $I$ is adjacent
        to all of $x,y,z$;
  \item each vertex of $T$ has a neighbor in $I$ and a neighbor strictly outside $T$.
\end{enumerate}
If $G$ has no separating triangle, every triangle bounds a face, every edge lies in at most two
triangles, and (ii) holds at every vertex.
\end{lemma}

\begin{proof}
An edge with an endpoint in $I$ lies in the closed disk bounded by $T$, so a triangle $C$ with a
vertex in $I$ lies in this disk and differs from $T$.  The inside of $C$ is then contained in the
inside of $T$, and a vertex of $T$ not on $C$ lies outside $C$.  If the inside of $C$ contained a
vertex, $C$ would be a separating triangle with fewer vertices inside than $T$.  So the inside of
$C$ contains no vertex, hence no edge (a triangle has no chord), and is a face.
An edge has two sides and at most one facial triangle on each, which proves (i).

If two nonconsecutive neighbors of $v\in I$ were adjacent, their triangle with $v$ would have
neighbors of $v$ on both sides, so it would not bound a face, contrary to (i).  This proves (ii).
A facial triangle on the interior side of an edge of $T$ is unique, which gives the first part
of (iii).  If $u\in I$ were adjacent to $x,y,z$, the three triangles $uxy,uyz,uzx$ would be facial
and would close the rotation at $u$, forcing $d(u)=3$.  If $x$ had no neighbor in $I$, then
$\{y,z\}$ would separate $I$ from $x$; if it had no neighbor outside $T$, then $\{y,z\}$ would
separate the exterior from $x$.  Both contradict $3$-connectivity, and the same holds for $y,z$.
Without a separating triangle, every triangle has an empty side, which is a face, and the same
arguments apply everywhere.
\end{proof}

From now on, when a separating triangle exists, $T$ and $I$ are as in
Lemma~\ref{lem:innermost}; vertices of $I$ are \emph{interior}, and faces inside $T$ are
\emph{interior faces}.  An interior vertex is \emph{eligible} if it has at most one neighbor on
$T$, and \emph{double-boundary} if it has exactly two.  By Lemma~\ref{lem:innermost}(iii) an
edge of $T$ has at most one common neighbor in $I$; we call it the \emph{interior tip} of that
edge.  The double-boundary vertices are exactly the interior tips.  Let $k_1$ and $k_2$ count
the interior vertices with exactly one and exactly two neighbors on $T$.  Then $k_2\le3$, and
there are $k_1+2k_2$ edges between $T$ and $I$.  If there is no separating triangle, every
vertex and every face is called interior (so $I=V(G)$), $V(T)$ and $E(T)$ are empty, every vertex
is eligible, and the statements below apply to the whole graph.

\begin{lemma}\label{lem:charge}
Give each interior vertex $v$ the charge $d(v)-6$ and each interior face $f$ the charge
$2|f|-6$.  Then
\begin{equation}\label{eq:charge}
  \sum_{v\in I}\bigl(d(v)-6\bigr)+\sum_{f\text{ inside }T}\bigl(2|f|-6\bigr)=-k_1-2k_2 .
\end{equation}
For the whole graph, with all vertices and faces, the same charges sum to $-12$.
\end{lemma}

\begin{proof}
Let $n=|I|$ and let $m$ be the number of edges with an endpoint in $I$.  The vertices and edges in
the closed disk of $T$ form a plane graph with $n+3$ vertices and $m+3$ edges.  It is connected,
because a path in $G$ from a vertex of $I$ to $T$ stays in the disk until it first meets $T$.  So
it has $m-n+1$ bounded faces, and these are exactly the interior faces.  The interior degrees sum to $2m-(k_1+2k_2)$, and the interior face lengths sum
to $2m+3$, since each edge with an endpoint in $I$ lies on two interior faces and each edge of
$T$ on one.  Hence the total is $2m-(k_1+2k_2)-6n+2(2m+3)-6(m-n+1)=-(k_1+2k_2)$.  If $G$ has
$N$ vertices, $M$ edges and $F$ faces, the total for the whole graph is $2M-6N+2\cdot2M-6F=-12$
by Euler's formula.
\end{proof}

\begin{lemma}\label{lem:facerule}
Let every interior face of length $\ell\ge4$ send $2-6/\ell$ to each incident interior vertex
(the \emph{face rule}).  Then every interior face $f$ retains exactly $b_f(2-6/\ell)\ge0$, where $b_f$ is the number of
vertices of $T$ on $f$, and every interior vertex $v$ has charge at least $d(v)-6\ge-2$.  Globally
every face retains zero.
\end{lemma}

\begin{proof}
Faces are bounded by cycles because $G$ is $3$-connected, so a face of length $\ell$ has
$\ell-b_f$ interior vertices and retains $(2\ell-6)-(\ell-b_f)(2-6/\ell)=b_f(2-6/\ell)$.  Vertices
only receive charge, and $d(v)\ge4$.
\end{proof}

\subsection{The octahedral cap and the equality case}

Let $xyz$ be a triangle.  Three vertices $A,B,C$ with
\[
  N(A)=\{B,C,x,y\},\qquad N(B)=\{A,C,y,z\},\qquad N(C)=\{A,B,x,z\}
\]
form an \emph{octahedral cap} on $xyz$; together with $x,y,z$ they span an octahedron.

\begin{lemma}\label{lem:octcap}
If $A,B,C$ form an octahedral cap on a triangle $xyz$, then every $D$-coloring of
$G-\{A,B,C\}$ extends to $G$, for every $D\ge6$.
\end{lemma}

\begin{proof}
Let $O_x$ be the set of neighbors of $x$ outside $\{x,y,z,A,B,C\}$, put $o_x=|O_x|$, and
$h_{xy}=|O_x\cap O_y|$; define $o_y,o_z,h_{yz},h_{xz}$ similarly.  Each $o$ and $h$ is at most $D-4$.
The uncolored edges are
\[
  U=(AB,AC,BC,Ax,Ay,By,Bz,Cx,Cz).
\]
We call $AB,AC,BC$ the \emph{inner} edges and the other six the \emph{arms}.  By
Lemma~\ref{lem:partner}, each inner edge has no colored incident edge and exactly two
colored partners, namely two edges of $xyz$ ($xy,yz$ for $AB$; $xy,xz$ for $AC$; $xz,yz$ for $BC$),
so its list has at least $2D-3$ colors.  The arm $Ax$ sees $2+o_x$ colored edges at $x$ and the
colored partners $yz$ and the $h_{xy}$ edges from $y$ to $O_x\cap O_y$; the other arms are
analogous:
\[
\begin{array}{c|cccccc}
 \text{arm} & Ax & Ay & By & Bz & Cx & Cz\\ \hline
 \text{incident} & 2+o_x & 2+o_y & 2+o_y & 2+o_z & 2+o_x & 2+o_z\\
 \text{partners} & 1+h_{xy} & 1+h_{xy} & 1+h_{yz} & 1+h_{yz} & 1+h_{xz} & 1+h_{xz}
\end{array}
\]
(common neighbors in $\{A,B,C\}$ contribute only uncolored edges).  So every arm has a list of
at least $2D-1-(3+2(D-4))=4$ colors.  The three inner edges conflict with all other edges of $U$.
Among the arms, by the last statement of Lemma~\ref{lem:partner}, the only nonconflicting
pairs are $\{Ax,Bz\}$, $\{Ay,Cz\}$ and $\{By,Cx\}$.  Color $Ax$ and $Bz$ arbitrarily from their
lists.  The arms $Ay,By,Cz,Cx$ form the cycle $Ay\,By\,Cz\,Cx\,Ay$ in the conflict graph and each
sees at most two colors on $Ax,Bz$, so they retain at least two colors each and are colored by
Lemma~\ref{lem:evencycle}.  Each inner edge then retains at least $2D-3-6\ge3$ colors, and the
inner triangle is colored greedily.
\end{proof}

\begin{lemma}\label{lem:equality}
Suppose $T$ exists, and charge is moved only between interior vertices and interior faces,
starting from the charges of Lemma~\ref{lem:charge}, so that finally:
\begin{enumerate}[label=\textup{(\alph*)}]
  \item every interior face $f$ of length $\ell$ has charge at least $b_f(2-6/\ell)$;
  \item every interior vertex with no neighbor on $T$ has nonnegative charge;
  \item every interior vertex with exactly one neighbor on $T$ has charge at least $-1$, and if
        its charge is $-1$ then it has no neighbor of degree four;
  \item every double-boundary vertex has charge at least $-2$, and if its charge is $-2$ then it
        has degree four and all its incident faces are triangles.
\end{enumerate}
Then the interior tips of the three edges of $T$ form an octahedral cap on $T$, so $G$ is not a
minimal counterexample.
\end{lemma}

\begin{proof}
The total charge is $-k_1-2k_2$ by \eqref{eq:charge}, and by (a)--(d) it is at least
$\sum_f b_f(2-6/\ell)-k_1-2k_2$.  Hence every face with a vertex of $T$ is a triangle, every vertex
with one neighbor on $T$ has charge exactly $-1$, and every double-boundary vertex has charge
exactly $-2$, so it has degree four.

Let $x\in T$ and list the neighbors of $x$ in $I$ as $p_1,\dots,p_r$ in rotation from $xy$ to
$xz$; $r\ge1$ by Lemma~\ref{lem:innermost}(iv).  The interior faces at $x$ are triangles, namely
$xyp_1$, $xp_ip_{i+1}$ ($1\le i<r$) and $xp_rz$.  So $p_1$ is the interior common neighbor of
$x,y$, and $p_r$ that of $x,z$; by Lemma~\ref{lem:innermost}(iii) $p_1\ne p_r$, so $r\ge2$, and
$p_1,p_r$ are double-boundary vertices of degree four.  If $r\ge3$, then $p_2$ is adjacent to
$x$ and $p_1$; it is not adjacent to $y$, since $p_1$ is the only interior common neighbor of
$x,y$, and not to $z$, since $p_r\ne p_2$ is the only one of $x,z$.  So $p_2$ has exactly one
neighbor on $T$, its charge is $-1$, and its neighbor $p_1$ has degree four, contrary to (c).
Hence the neighbors of $x$ in $I$ are exactly the two tips at $x$, and they are adjacent.

The same holds at $y$ and at $z$.  As every interior vertex with a neighbor on $T$ is therefore
a tip, $k_1=0$.  Let $A,B,C$ be the interior tips of $xy,yz,zx$; they are distinct by
Lemma~\ref{lem:innermost}(iii), pairwise adjacent by the previous paragraph, and of degree four,
so $N(A)=\{B,C,x,y\}$, $N(B)=\{A,C,y,z\}$ and $N(C)=\{A,B,x,z\}$.  Lemma~\ref{lem:octcap} gives
a contradiction to minimality.
\end{proof}

\section{Single-vertex reductions and maximum degree at least eight}\label{sec:star}

\begin{lemma}\label{lem:star}
Let $v$ be an eligible vertex.  Delete $v$ and, for $u\in N(v)$, put $r_u=|N(v)\cap N(u)|$.  Then
\begin{equation}\label{eq:star}
  |L(vu)|\ge 2D-d(u)-2r_u ,
\end{equation}
and the uncolored edges, the spokes at $v$, form a clique in the conflict graph.
\end{lemma}

\begin{proof}
Exactly $d(u)-1$ colored edges at $u$ see $vu$.  Fix $t\in N(v)\cap N(u)$.  A partner $tq$ of
$vu$ with $q\in N(v)$ gives a common neighbor $q\ne u$ of $v$ and $t$, and one with $q\in N(u)$
gives a common neighbor $q\ne v$ of $u$ and $t$.  The edge $vt$ touches $I$, and so does $ut$,
because $v$ has at most one neighbor on $T$; by Lemma~\ref{lem:innermost}(i) each of them lies in
at most two triangles, one of which is $vut$.  So each $t$ contributes at most two colored
partners, which proves \eqref{eq:star}.  Without $T$ the same holds globally.
\end{proof}

\begin{lemma}\label{lem:starcons}
Let $v$ be an eligible vertex of a minimal counterexample.  The link of $v$ has no isolated
vertex and is a subgraph of the rotation cycle at $v$; so it is $C_4$, $P_4$ or $2K_2$ if
$d(v)=4$, and $C_5$, $P_5$ or $P_3+P_2$ if $d(v)=5$.  Moreover the following do not occur:
\begin{enumerate}
  \item $d(v)=4$ and the link of $v$ is $P_4$ or $2K_2$ (any $D\ge6$);
  \item $d(v)=5$ and the link of $v$ is $P_5$ or $P_3+P_2$ (when $D\ge7$);
  \item $d(v)\in\{4,5\}$ (when $D\ge9$);
  \item $d(v)=4$, or $d(v)=5$ and $v$ has a neighbor of degree at most seven (when $D=8$);
  \item $d(v)=4$ and $v$ has a neighbor of degree at most six; or $d(v)=5$ and $v$ has a
        neighbor of degree at most five and another neighbor of degree at most six (when $D=7$).
\end{enumerate}
In particular, for $D=7$ every eligible four-vertex has link $C_4$ and four neighbors of degree
seven, every eligible five-vertex has link $C_5$, and if an eligible five-vertex has a neighbor of
degree at most five, then its other four neighbors have degree seven.
\end{lemma}

\begin{proof}
Every edge $vu$ lies in a triangle (Lemma~\ref{lem:triangle}), whose third vertex is a link
neighbor of $u$; the link is a subgraph of the rotation cycle by Lemma~\ref{lem:innermost}(ii).
This gives the list of possible links, and $r_u$ is the degree of $u$ in the link.  In each case
below Lemma~\ref{lem:hall} colors the clique of spokes, because the $i$-th smallest lower bound
from \eqref{eq:star} is at least $i$.

(i) For $P_4$ the bounds dominate $(D-2,D-4,D-4,D-2)$, which is at least $(4,2,2,4)$; for
$2K_2$ all four bounds are at least $D-2\ge4$.  (ii) For $P_5$ the bounds dominate
$(D-2,D-4,D-4,D-4,D-2)\ge(5,3,3,3,5)$; for $P_3+P_2$, where only one link vertex has link degree
two, they dominate $(D-2,D-4,D-2,D-2,D-2)\ge(5,3,5,5,5)$.  (iii) Since $r_u\le2$ and $d(u)\le D$,
every spoke has at least $D-4\ge5$ colors.  (iv) For $D=8$ every spoke has at least four colors,
which handles degree four; at a five-vertex the spoke to a neighbor of degree at most seven has
at least five colors.  (v) By (i) and (ii) only the links $C_4$ and $C_5$ remain.  For $D=7$ and
link $C_4$ the bounds are $10-d(u)\ge3$, and a neighbor of
degree at most six gives a fourth bound of at least four.  For link $C_5$ the bounds are
$10-d(u)$, and the two named neighbors give bounds of at least five and four, so the bounds
dominate $(3,3,3,4,5)$.  The final assertions combine (i), (ii) and (v).
\end{proof}

\begin{theorem}\label{thm:D9}
Theorem~\ref{thm:main} holds for every $D\ge9$.
\end{theorem}

\begin{proof}
Suppose a minimal counterexample exists.  By Lemma~\ref{lem:starcons}(iii) every eligible vertex
has degree at least six.  Without a separating triangle all vertex charges $d(v)-6$ and all face
charges $2|f|-6$ are nonnegative, contradicting the total $-12$.  Otherwise apply the face rule
(Lemma~\ref{lem:facerule}).  Faces retain $b_f(2-6/\ell)$, eligible vertices are nonnegative, and a
double-boundary vertex has charge at least $d(v)-6\ge-2$, with equality only if $d(v)=4$ and it
receives nothing, that is, all its incident faces are triangles.  Lemma~\ref{lem:equality}
gives a contradiction.
\end{proof}

\begin{theorem}\label{thm:D8}
Theorem~\ref{thm:main} holds for $D=8$.
\end{theorem}

\begin{proof}
Suppose a minimal counterexample exists.  By Lemma~\ref{lem:starcons}(iv) there is no eligible
four-vertex and every eligible five-vertex has five neighbors of degree eight.  Apply the face
rule, and let every interior eight-vertex send $\tfrac14$ to each adjacent interior five-vertex
(globally: every eight-vertex to every adjacent five-vertex).  An eight-vertex sends at most
$2=8-6$ and stays nonnegative.  An eligible five-vertex has at least four interior neighbors,
all of degree eight, so it receives at least $1$ and is nonnegative.  Six- and seven-vertices are
nonnegative.  A double-boundary vertex of degree four or five only receives, and one of degree at
least six stays nonnegative; so it has charge at least $-2$, with equality only for degree four
and triangular incident faces.  Faces retain $b_f(2-6/\ell)$.  Without a separating triangle every
vertex and face is nonnegative, contradicting $-12$; otherwise Lemma~\ref{lem:equality} gives a
contradiction.
\end{proof}

\section{Maximum degree seven}\label{sec:d7}

In this section $D=7$, so the palette has thirteen colors.  We show that three configurations
are reducible and then add one discharging rule.

\begin{lemma}\label{lem:H4}
Let $v$ be an interior four-vertex with link $u_0u_1u_2u_3u_0$ (indices mod $4$) and no
neighbor on $T$.  Take a $D$-coloring of $G-v$ and uncolor the four rim edges $e_i=u_iu_{i+1}$.
The resulting partial coloring extends to $G$.
\end{lemma}

\begin{proof}
The uncolored edges are the spokes $vu_i$ and the rim edges $e_i$.  The spokes form a $K_4$.
Each rim edge meets two spokes and, by Lemma~\ref{lem:partner}, is a partner of the other two
($e_i$ and $vu_{i+2}$ have the cross edges $u_iv,u_{i+1}v,u_{i+1}u_{i+2}$).  Consecutive rim edges
are incident, while opposite rim edges do not conflict, since $u_0u_2,u_1u_3$ are not edges.
Since a diamond containing two disjoint edges has exactly their four endpoints, these are all
the conflicts: the conflict graph is $K_4\vee C_4$.

All $u_i$ are interior, so every edge $u_iu_{i+1}$ and every edge from some $u_i$ lies in at most
two triangles.  A spoke $vu_i$ sees at most $7-3=4$ colored edges at $u_i$, and for each
$t\in\{u_{i-1},u_{i+1}\}$ at most one colored partner (the edge from $t$ to the second common
neighbor of $u_it$; the partner through $v$ is an uncolored rim edge).  So spoke lists have size
at least $13-4-2=7$.  A rim edge $e_i$ sees at most $4+4$ colored edges at $u_i,u_{i+1}$; its
common neighbors are $v$, which contributes only uncolored edges, and at most one further vertex
$w$, which contributes at most two partners (from the sides $u_iw$ and $u_{i+1}w$).  So rim lists
have size at least $13-8-2=3$.  We now work with the actual lists $L$.

If two opposite rim edges have a common color $\gamma$ in their lists, give both color $\gamma$.
The other two rim edges are opposite, hence nonconflicting, and each has at least two colors
other than $\gamma$; color them.  The rim uses at most three colors, so each spoke keeps at least
four colors and the spokes are colored greedily.

Otherwise the lists of both opposite pairs are disjoint.  Color the rim greedily from its lists;
the four colors are distinct, since consecutive rim edges are incident and opposite ones have
disjoint lists.  Each spoke keeps at least three colors, and Hall's condition for the four spokes
can fail only if they all keep the same three-set $B$.  Then every spoke list, having at least
seven colors, equals the seven-set $A$ consisting of $B$ and the four rim colors.  Every colored
edge at $u_i$ sees $vu_i$, so its color is not in $A$.  Hence the colored edges at the ends of a
rim edge remove no color of $A$ from its list, and its at most two colored partners remove at
most two: each rim list contains at least five colors of $A$.  Two five-subsets of a seven-set
intersect, contradicting the disjointness of opposite rim lists.
\end{proof}

\begin{remark}
The list sizes $(7,7,7,7,3,3,3,3)$ alone do not suffice.  Give every spoke the list
$\{1,\dots,7\}$ and the rim edges, in cyclic order, the lists $\{1,2,3\},\{1,2,3\},\{4,5,6\},
\{4,5,6\}$; then eight distinct colors would be needed from seven.  The proof uses that all
eight lists come from one coloring of $G-v$.
\end{remark}

\begin{lemma}\label{lem:H5}
Let $v$ be an eligible five-vertex with link $C_5=u_0u_1u_2u_3u_4u_0$, and suppose that $u_0$ and
$u_1$ are interior and have degree at most six.  Take a $D$-coloring of $G-v$ and uncolor
$e=u_0u_1$.  The resulting partial coloring extends to $G$.
\end{lemma}

\begin{proof}
The five spokes form a $K_5$.  The edge $e$ meets $vu_0,vu_1$, is a partner of $vu_2$ and $vu_4$,
and does not conflict with $vu_3$, because $u_3$ is adjacent to neither $u_0$ nor $u_1$.  So the
conflict graph is $K_4\vee 2K_1$ with universal vertices $vu_0,vu_1,vu_2,vu_4$ and the
nonadjacent pair $vu_3,e$.  Before $e$ is uncolored, the proof of Lemma~\ref{lem:star} bounds the
colored edges seeing $vu_i$ by $d(u_i)-1$ incident edges and four partners, giving
$|L(vu_i)|\ge 10-d(u_i)$; and $e$ is one of these counted edges for $i\in\{0,1,2,4\}$ (an
incident edge for $i=0,1$, and the partner through $u_1$, respectively $u_0$, for $i=2,4$).  So
after uncoloring $e$ these four bounds increase by one.  The edge $e$
sees at most $4+4$ colored edges at $u_0,u_1$; its common neighbors are $v$ and at most one further
vertex $w$, which contributes at most two partners because $u_0w$ and $u_1w$ touch $I$.  Hence, in
the order $(vu_0,vu_1,vu_2,vu_3,vu_4,e)$, the lists have sizes at least $(5,5,4,3,4,3)$.

If the lists of $vu_3$ and $e$ share a color, give both that color; the four universal spokes keep
at least $(4,4,3,3)$ colors and are colored greedily.  If the two lists are disjoint, we check
Hall's condition for all six lists: a set containing both $vu_3$ and $e$ has a union of size at
least $6$; a set of at most four elements containing a universal spoke has a union of size at
least $4$; a five-element set with exactly one of $vu_3,e$ contains $vu_0$, whose list has five
colors; and the remaining sets are trivial.  Distinct representatives give a valid coloring.
\end{proof}

\begin{lemma}\label{lem:H55}
No two eligible five-vertices with full $C_5$ links are adjacent.
\end{lemma}

\begin{proof}
Suppose $a,b$ are such vertices, with cyclic links $b,x,p,q,y$ at $a$ and $a,x,r,s,y$ at $b$.  The
edge $ab$ has exactly the common neighbors $x,y$, so the eight named vertices are distinct.
Delete $a$ and $b$; in the order
\[
  U=(ab,ax,ay,bx,by,ap,aq,br,bs)
\]
the numbers of colored incident edges and of colored partners are at most
\[
  (0,5,5,5,5,6,6,6,6)\qquad\text{and}\qquad(4,2,2,2,2,3,3,3,3),
\]
respectively, so the lists have sizes at least $(9,6,6,6,6,4,4,4,4)$.  Indeed, the colored partners of $ab$
are $xp,xr,yq,ys$.  For $ax$, the common neighbors are $b$ (only uncolored edges) and $p$, which
contributes $pq$ and at most one further partner through the side $xp$; and $x$ has at most five
colored edges.  The edges $ay,bx,by$ are analogous with $(pq,yq)$, $(rs,xr)$, $(rs,ys)$.  For $ap$,
the common neighbors are $x$ and $q$: $x$ contributes $xb$ (uncolored) and at most one edge
through the side $xp$, and $q$ contributes $qy$ and at most one edge through the side $pq$; the
same holds for $aq,br,bs$.  Every side used here joins two neighbors of an eligible vertex, so it
touches $I$ and lies in at most two triangles.

The edge $ab$ meets every other edge of $U$.  The shared arms $ax,ay,bx,by$ form a $K_4$, the
private arms form the two incident pairs $\{ap,aq\}$ and $\{br,bs\}$, and no arm of one pair
conflicts with an arm of the other (their only possible cross edges are $ab$ and one edge between
private vertices).  Each private arm misses exactly one shared arm:
\[
  ap\not\sim by,\qquad aq\not\sim bx,\qquad br\not\sim ay,\qquad bs\not\sim ax,
\]
and sees the other three.

Suppose some list of $\{ap,aq\}$ meets some list of $\{br,bs\}$, and give these two arms a common
color $\alpha$.  Color the other two private arms avoiding $\alpha$; they do not conflict and
each has at least three choices.  Each shared arm sees at most three private colors, and the two
shared arms that see both $\alpha$-colored arms see at most two.  So the shared $K_4$ keeps lists of
sizes at least $(3,3,4,4)$ and is colored greedily.  Otherwise all lists of $\{ap,aq\}$ are disjoint
from all lists of $\{br,bs\}$, and Hall's condition holds for the eight arms: at most four arms have
a union of at least four colors; five or six arms include a shared arm with six colors; seven or
eight arms include a private arm from each pair, with disjoint lists of at least four colors each.
In both cases color $ab$ last; it sees at most eight colors and has at least nine.
\end{proof}

\begin{theorem}\label{thm:D7}
Theorem~\ref{thm:main} holds for $D=7$.
\end{theorem}

\begin{proof}
Suppose a minimal counterexample exists.  Apply the face rule, and in addition let every
interior seven-vertex send $\tfrac13$ to each adjacent eligible vertex of degree four or five.  For
an eligible vertex $v$ let $b(v)\in\{0,1\}$ be its number of neighbors on $T$ ($b(v)=0$ if there is
no separating triangle).

By Lemma~\ref{lem:starcons}, every recipient has link $C_4$ or $C_5$, so all its faces are
triangles, and a recipient of degree four has only neighbors of degree seven.  If two recipients
$w,w'$ were consecutive in the rotation at a donor $u$, the face between $uw$ and $uw'$ would be
incident with $w$, hence a triangle, so $w$ and $w'$ would be adjacent.  Then neither would have
degree four, so they would be adjacent eligible five-vertices with full links, contrary to
Lemma~\ref{lem:H55}.  So $u$ has at
most three recipients and finishes with charge at least $1-1=0$.

By Lemma~\ref{lem:H4}, an eligible four-vertex $v$ has $b(v)=1$; its other three neighbors are
interior seven-vertices, so it finishes with charge exactly $-1$, and all its neighbors have
degree seven.  Let $v$ be an eligible five-vertex and $b=b(v)$.  If $v$ has a neighbor of degree at
most five, its other four neighbors have degree seven (Lemma~\ref{lem:starcons}), and at least
$4-b$ of them are interior.  Otherwise its neighbors have degree six or seven, and by
Lemma~\ref{lem:H5} no two consecutive vertices of its link $C_5$ are interior six-vertices.  The
interior neighbors form the whole link if $b=0$ and a path on four vertices if $b=1$, so at most
two of them have degree six and at least $3-b$ are seven-vertices.  In both cases $v$ finishes
with charge at least $-1+(3-b)/3=-b/3$.

Interior six-vertices are nonnegative, and faces retain $b_f(2-6/\ell)$.  A double-boundary vertex
receives nothing from seven-vertices and sends nothing unless it is a donor; so it finishes with
charge at least $-2$, with equality only for degree four and triangular incident faces.  Without a
separating triangle all final charges are nonnegative, contradicting the total $-12$.  Otherwise
the hypotheses of Lemma~\ref{lem:equality} hold: a vertex with one neighbor on $T$ finishes with
charge at least $-1$, with equality only for four-vertices, whose neighbors all have degree seven.
The resulting octahedral cap contradicts minimality.
\end{proof}

\section{Maximum degree six}\label{sec:d6}

In this section $D=6$ and the palette has eleven colors.  After two more lemmas we find a facial
triangle containing a vertex of negative charge (Lemma~\ref{lem:selection}), and we show that
such a triangle is reducible by checking finitely many cases by computer
(Lemmas~\ref{lem:generic} and~\ref{lem:onespine}).

\subsection{Two more lemmas}

\begin{lemma}\label{lem:cap4}
Let $xyz$ be a triangle and let $A,B,C,v$ be vertices with
\[
  N(A)=\{B,v,x,y\},\quad N(B)=\{A,C,v,y,z\},\quad N(C)=\{B,v,x,z\},\quad N(v)=\{A,B,C,x\}.
\]
Then every $D$-coloring of $G-\{A,B,C,v\}$ extends to $G$, for every $D\ge6$.
\end{lemma}

\begin{figure}[ht]
\centering
\includegraphics{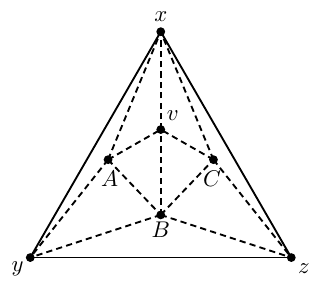}
\caption{The configuration of Lemma~\ref{lem:cap4}.  The dashed edges are the twelve edges
of $U$.}
\label{fig:cap4}
\end{figure}

\begin{proof}
See Figure~\ref{fig:cap4}.  Define $o_x,o_y,o_z$ and $h_{xy},h_{xz},h_{yz}$ as in the proof of Lemma~\ref{lem:octcap}, now with
respect to the seven named vertices.  Then $o_x\le D-5$, $o_y,o_z\le D-4$, $h_{xy},h_{xz}\le D-5$
and $h_{yz}\le D-4$.  Order the uncolored edges as
\[
  U=(AB,BC,Av,Cv,Bv,Ax,Ay,By,Bz,Cz,Cx,vx).
\]
By Lemma~\ref{lem:partner}, the colored incident edges and colored partners are:
\[
\begin{array}{c|cc}
 \text{edge} & \text{incident} & \text{partners}\\ \hline
 AB,\ BC & 0 & 2\\
 Av,\ Cv & 0 & 1\\
 Bv & 0 & 0\\
 Ax & 2+o_x & 1+h_{xy}
\end{array}
\qquad
\begin{array}{c|cc}
 \text{edge} & \text{incident} & \text{partners}\\ \hline
 Ay & 2+o_y & 1+h_{xy}\\
 By & 2+o_y & 1+h_{yz}\\
 Bz & 2+o_z & 1+h_{yz}\\
 Cz & 2+o_z & 1+h_{xz}
\end{array}
\qquad
\begin{array}{c|cc}
 \text{edge} & \text{incident} & \text{partners}\\ \hline
 Cx & 2+o_x & 1+h_{xz}\\
 vx & 2+o_x & 0\\
 & & \\
 & &
\end{array}
\]
For instance, the only colored partner of $Av$ is $xy$, the common neighbors of $v$ and $x$ are $A$
and $C$, whose edges are uncolored, and for every other arm the partner counted by the $1$ is the
edge of $xyz$ opposite the arm's end in $\{x,y,z\}$.  The lists therefore have sizes at least
\[
  (2D-3,\,2D-3,\,2D-2,\,2D-2,\,2D-1,\,6,\,5,\,4,\,4,\,5,\,6,\,D+2),
\]
which dominates $(9,9,10,10,11,6,5,4,4,5,6,8)$.  Call the first five edges of $U$ \emph{inner}
and the other seven \emph{arms}.  By the last statement of Lemma~\ref{lem:partner}, the pairs
of arms
\[
  \{Ax,Bz\},\ \{Bz,vx\},\ \{vx,By\},\ \{By,Cx\},\ \{Cx,Ay\},\ \{Ay,Cz\},\ \{Cz,Ax\}
\]
do not conflict, and neither do the pairs
\[
  \{AB,Cx\},\ \{AB,Cz\},\ \{BC,Ax\},\ \{BC,Ay\},\ \{Av,Bz\},\ \{Av,Cz\},\ \{Cv,Ay\},\ \{Cv,By\}
\]
(each has at most two cross edges).  So every arm conflicts with at most four other arms, and the
inner edges conflict with at most $5,5,5,5,7$ arms, respectively.  Color $By$ and $Bz$ with
distinct colors from their lists, of size at least four, and then the other five arms greedily:
each has at least five colors and at most four colored conflicting arms.  The inner edges then
keep lists of sizes at least $(4,4,5,5,4)$.  They form a clique ($AB,Cv$ and $BC,Av$ have three
cross edges each, and all other pairs are incident), and the sorted sizes $(4,4,4,5,5)$ satisfy
Lemma~\ref{lem:hall}.  Larger $D$ only enlarges the lists.
\end{proof}

\begin{lemma}\label{lem:exterior}
Let $T=xyz$ be a separating triangle of $G$ and $O$ the set of vertices strictly outside $T$.
Then $N(x)\cap O\ne N(y)\cap O$, and similarly for the pairs $y,z$ and $z,x$.  Consequently, if
$|N(x)\cap O|\le A_x$ and $|N(y)\cap O|\le A_y$, then
\begin{equation}\label{eq:extbound}
  |N(x)\cap N(y)\cap O|\le\min\bigl(A_x,A_y,\max(A_x,A_y)-1\bigr).
\end{equation}
\end{lemma}

\begin{proof}
Suppose $N(x)\cap O=N(y)\cap O=W$.  If $W$ were empty, $O$ would attach to the rest of $G$ only at
$z$, which would be a cut vertex; so $W\ne\emptyset$.  A component of $G[O\setminus W]$ has no
neighbor inside $T$
and no neighbor $x$ or $y$, so it attaches only to $W\cup\{z\}$, all of whose vertices are common
neighbors of $x$ and $y$.  Three attachments would give, after contracting the component, a
$K_{3,3}$ minor against $x,y$; at most two attachments would form a separating set of size at most
two.  Hence $O=W$.  The join of $xy$ with $G[W\cup\{z\}]$ is a planar subgraph, so $G[W\cup\{z\}]$ is
a linear forest (Lemma~\ref{lem:join}); it is connected, since otherwise deleting $x,y$ would cut a
component avoiding $z$ from the inside of $T$.  So it is a path with at least two vertices, one of
its ends lies in $W$, and that end has exactly the neighbors $x$, $y$ and one path neighbor,
contradicting minimum degree four.  For \eqref{eq:extbound}: the intersection has size at most
$A_x$ and $A_y$, and since the two sets differ, one of them has at least one more element than the
intersection.
\end{proof}

\subsection{Selecting a facial triangle}

Apply only the face rule (Lemma~\ref{lem:facerule}).  An eligible vertex of degree at least six
is nonnegative, and so is an eligible five-vertex with link $P_3+P_2$, since two of its incident
faces have length at least four and give at least $\tfrac12$ each.  By
Lemma~\ref{lem:starcons}(i), a negative eligible vertex therefore has link $C_4$, $C_5$ or $P_5$.

\begin{lemma}\label{lem:selection}
Every negative eligible vertex $v$ lies in a facial triangle $Q=vpq$ with $p,q\in I$ in which at
most one vertex is double-boundary.  If there is no separating triangle, $Q$ has no
double-boundary vertex.
\end{lemma}

\begin{proof}
Let $J$ be the link of $v$, and let $P$ be the set of interior neighbors of $v$ that are
double-boundary.  Each member of $P$ is the unique interior common neighbor of an edge of $T$, so
$|P|\le3$, and $P=\emptyset$ without a separating triangle.  Call an edge $pq$ of $J$ \emph{good}
if $p,q\notin V(T)$ and at most one of $p,q$ lies in $P$.  A good edge gives the required triangle:
$p,q\in I$, the triangle $vpq$ is facial by Lemma~\ref{lem:innermost}(i), and $v$, being eligible,
is not double-boundary.

Suppose that no edge of $J$ is good.  If $v$ has no neighbor on $T$, both ends of every edge of $J$
lie in $P$; as $J$ has no isolated vertex, $|P|\ge d(v)\ge4$, a contradiction.  Otherwise $v$ has a
unique neighbor $t$ on $T$, and every vertex that is not isolated in $J-t$ lies in $P$.  A member of
$P$ not adjacent to $t$ is the interior tip of the edge of $T$ opposite $t$; hence
\begin{equation}\label{eq:Pt}
  |P\setminus N_J(t)|\le1 .
\end{equation}
If $J=C_5$, or $J=P_5$ with $t$ an end or the middle vertex, then $J-t$ has four nonisolated
vertices, too many for $P$.  If $J=P_5$ and $t$ is next to an end, the nonisolated vertices of
$J-t$ form a path on three vertices, two of which are not adjacent to $t$, contrary to
\eqref{eq:Pt}.  So $J=C_4$, with cyclic order $t,A,B,C$, and $P=\{A,B,C\}$.

Write $t=x$.  The vertices $A$ and $C$ are adjacent to $x$, and their second neighbors on $T$ are
distinct, because an edge of $T$ has only one interior common neighbor; call them $y$ and $z$, so
that $A$ is the interior tip of $xy$ and $C$ that of $xz$.  The edge $xB$ is absent, being a chord
of $J$, so $B$ is the interior tip of $yz$.  All triangles touching $I$ are faces.  At $A$ the faces
$xyA,yBA,ABv,Avx$ close the rotation; at $B$ the faces $yzB,zCB,BCv,BvA,BAy$ do; at $C$ the faces
$zxC,xvC,CvB,BCz$ do.  Hence $N(A)=\{x,y,B,v\}$, $N(B)=\{y,z,C,v,A\}$, $N(C)=\{z,x,v,B\}$ and
$N(v)=\{x,A,B,C\}$, which is the configuration of Lemma~\ref{lem:cap4}; this contradicts
minimality.
\end{proof}

\subsection{Templates}

Lemmas~\ref{lem:normal}--\ref{lem:monotone} hold for every $D\ge6$; we use them only for $D=6$.
Let $Q=\{a,b,c\}$ be a facial triangle with $a,b,c\in I$ and degrees $d_a,d_b,d_c\ge4$.  For
$q\in Q$, the $d_q-2$ neighbors of $q$ outside $Q$ are its \emph{outer neighbors}; we list them in
rotation order and write $p_{qr}[q]$ for the one next to the edge $qr$.  For each pair $qr$ of
vertices of $Q$, either $p_{qr}[q]=p_{qr}[r]$, and we call this vertex the \emph{shared tip}
$p_{qr}$, or the two are distinct, and we say that $qr$ has \emph{split tips}.  The vector
$\sigma=(\sigma_{ab},\sigma_{ac},\sigma_{bc})\in\{0,1\}^3$ records which pairs have a shared tip.
Thus the outer neighbors are, in order,
\begin{equation}\label{eq:paths}
\begin{aligned}
  a&:\ p_{ac}[a],a_0,\dots,a_{d_a-5},p_{ab}[a],\\
  b&:\ p_{ab}[b],b_0,\dots,b_{d_b-5},p_{bc}[b],\\
  c&:\ p_{bc}[c],c_0,\dots,c_{d_c-5},p_{ac}[c],
\end{aligned}
\end{equation}
where the middle part is empty for degree four.  An edge joining two consecutive outer neighbors
of $q$ is a \emph{rim edge} of $q$; there are $d_a+d_b+d_c-9$ potential rim edges.  Given the
degrees, $\sigma$ and a set of potential rim edges, the \emph{template} $R$ is the graph
consisting of $ab,ac,bc$, all edges from each $q\in Q$ to its outer neighbors, and the chosen rim
edges; the template of $Q$ itself uses the rim edges present in $G$ (Figure~\ref{fig:template}).
Let $U$ be the set of edges of $R$ with an end in $Q$; for the template of $Q$ these are the edges
of $G$ with an end in $Q$.  Let $H_R$ be the graph on $U$ in which two edges are adjacent if they
meet or if at least three of their four cross edges lie in $R$.

\begin{figure}[ht]
\centering
\includegraphics{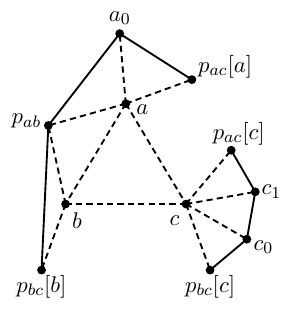}
\caption{A template with $(d_a,d_b,d_c)=(5,4,6)$, a shared tip at $ab$ and split tips at $ac$ and
$bc$.  The dashed edges form $U$; the solid edges are the $d_a+d_b+d_c-9=6$ potential rim edges,
all present here.}
\label{fig:template}
\end{figure}

\begin{lemma}\label{lem:normal}
Let $Q\subseteq I$ be a facial triangle whose vertices have degree at least four, and let $R$ be
its template.  The vertices named in \eqref{eq:paths} are distinct except for shared tips, and
for each $q\in Q$ the edges of $G$ between neighbors of $q$ are exactly the edges of $R$ between
neighbors of $q$.  Moreover, $H_R$ is the conflict graph on $U$.
\end{lemma}

\begin{proof}
A vertex that is an outer neighbor of two vertices $q,r\in Q$ is a common neighbor of $q$ and $r$
other than the third vertex of $Q$.  There is at most one such vertex, since $qr$ lies in at most
two triangles (Lemma~\ref{lem:innermost}(i)), and its triangle with $q$ and $r$ is a face, so it
is the outer neighbor of $q$ next to $qr$ and also that of $r$: it is a shared tip.  Two shared
tips cannot coincide, since such a vertex would be adjacent to all of $Q$ and would create a
triangle in the link of a vertex of degree at least four, whereas links are subgraphs of rotation
cycles.  The link of $q$ is a subgraph of its rotation cycle, so its only possible edges are the
edge between the other two vertices of $Q$, the edges $r\,p_{qr}[q]$ for the two vertices
$r\in Q\setminus\{q\}$ (present exactly for shared tips), and rim edges of $q$.  This proves the
first part.

Let $e,e'\in U$ be disjoint and choose $q\in e\cap Q$, $q'\in e'\cap Q$, writing $e=qs$, $e'=q's'$.
By Lemma~\ref{lem:partner} the pair conflicts if and only if at least three of $qq',qs',sq',ss'$
are edges of $G$.  Here $q\ne q'$, so $qq'\in R$, and $qs'$, $sq'$ are edges of $G$ exactly when
they are edges of $R$, since $R$ contains all edges at $q$ and $q'$.  The edge $ss'$ affects the
outcome only when exactly one of $qs',sq'$ is present; then $s$ and $s'$ are both neighbors of
$q$, or both neighbors of $q'$, so $ss'$ is an edge of $G$ exactly when it is an edge of $R$.
Pairs that meet are adjacent both in $H_R$ and in the conflict graph.
\end{proof}

We delete $Q$ and keep the colors of all other edges, so $U$ is the set of uncolored edges.

\begin{lemma}\label{lem:counts}
For an edge of $U$ with both ends in $Q$ there are no colored incident edges, and all its
partners are edges of $R$; so its list has size at least $2D-1$ minus the number of its partners
in $R$ that are not in $U$.  For an arm $qx$ with $q\in Q$ and $x\notin Q$ put
\begin{align*}
  m(x)&=|N_R(x)\cap Q|,\qquad C(qx)=\bigl(N_R(q)\cap N_R(x)\bigr)\setminus Q,\\
  K(qx)&=\{ty\in E(R):\ t\in C(qx),\ y\in N_R(q)\setminus(Q\cup\{x,t\})\}.
\end{align*}
If no edge $xt$ with $t\in C(qx)$ is an edge of $T$, the list of $qx$ has size at least
\begin{equation}\label{eq:arm}
  f(qx)=2D-1-\bigl(D-m(x)\bigr)-|K(qx)|-|C(qx)| .
\end{equation}
For each $t\in C(qx)$ with $xt\in E(T)$, let $h_{xt}$ be the number of common neighbors of $x,t$
strictly outside $T$; then subtracting $h_{xt}$ from \eqref{eq:arm} gives a valid bound.
\end{lemma}

\begin{proof}
A partner of an edge inside $Q$ joins two neighbors of one of its ends, so it is an edge of $R$
by Lemma~\ref{lem:normal}.  For an arm $qx$, all edges at $q$ are uncolored and exactly $m(x)$
edges at $x$ are uncolored, so at most $D-m(x)$ colored incident edges see $qx$.  The common
neighbors of $q$ and $x$ are the neighbors of $x$ in the link of $q$, which $R$ determines; those
in $Q$ have only uncolored edges.  For $t\in C(qx)$, the partners $ty$ with $y\in N(q)$ are edges
of the link of $q$ and are counted in $K(qx)$.  Any further partner $ty$ has
$y\in N(x)\setminus N(q)$, so $y\ne q$ is a common neighbor of $x$ and $t$.  Since $x,t\in N(q)$
and $q\in I$, both lie in $I\cup V(T)$, so either $xt$ has an end in $I$ or $xt\in E(T)$.  If $xt$
has an end in $I$, there is at most one such $y$, which accounts for the term $|C(qx)|$.  If
$xt\in E(T)$, its common neighbors other than $q$ are the third vertex of $T$ and $h_{xt}$
vertices strictly outside $T$; the term $|C(qx)|$ accounts for the former.  Counting a partner
twice only weakens the bound.  Lemma~\ref{lem:lists} completes the proof.
\end{proof}

\begin{lemma}\label{lem:monotone}
Fix the degrees and $\sigma$, let $R_0$ be a template, and let $R$ be obtained from $R_0$ by
deleting rim edges that are not edges of $T$.  Compute the bounds $f_0$ and $f$ of
Lemma~\ref{lem:counts} for $R_0$ and $R$ with the same subtractions at the same arms.  Then $U$ is
unchanged, $H_R\subseteq H_{R_0}$, and $f\ge f_0$ coordinatewise.  Hence if $H_{R_0}$ is
colorable from all lists of sizes at least $f_0$, then $H_R$ is colorable from all lists of sizes
at least $f$.
\end{lemma}

\begin{proof}
The edges at $a$, $b$ and $c$, hence $U$ and $m(x)$, are unchanged.  Deleting a rim edge can
only shrink $C(qx)$, $K(qx)$ and the set of partners in $R$ of an edge inside $Q$, and it removes
conflicts but creates none.  Apply Lemma~\ref{lem:greedy}.
\end{proof}

\subsection{Generic templates}

A template with no edge of $T$ among its rim edges is \emph{generic}.  By
Lemma~\ref{lem:selection}, the triangles we must reduce at $D=6$ contain a vertex of degree at
most five, and all their degrees lie in $\{4,5,6\}$.  Permuting $a,b,c$ permutes the entries of
$\sigma$ and reverses some of the lists in \eqref{eq:paths}, so we may assume that the degrees
are sorted:
\begin{equation}\label{eq:triples}
  (d_a,d_b,d_c)\in\{444,\,445,\,446,\,455,\,456,\,466,\,555,\,556,\,566\}.
\end{equation}
A template is \emph{full} if all potential rim edges are present.  With the eight choices of
$\sigma$ there are $72$ full generic templates.

\begin{lemma}\label{lem:generic}
For each of the $72$ full generic templates $R$ at $D=6$, the graph $H_R$ is colorable from all
lists of sizes at least the bounds of Lemma~\ref{lem:counts}.  This is checked by computer: for
$59$ of them we give a removal order as in Lemma~\ref{lem:greedy}, and for the other $13$ a
nonzero coefficient as in Lemma~\ref{lem:cn} (Appendix~\ref{app:comp}).  Consequently, at
$D=6$, if $Q$ has a vertex of degree at most five and its template is generic, every coloring of
$G-Q$ extends to $G$.
\end{lemma}

\begin{proof}
Lemma~\ref{lem:monotone}, applied with no subtractions, carries colorability from the full
template to every generic template with the same degrees and $\sigma$, and
Lemma~\ref{lem:lists} turns a list coloring into an extension.
\end{proof}

\subsection{Triangles with a double-boundary vertex}

Suppose that exactly one vertex $q^\ast$ of $Q$ is double-boundary, with neighbors $x,y$ on $T$,
and let $z$ be the third vertex of $T$ and $O$ the set of vertices strictly outside $T$.  Then
$q^\ast xy$ is a facial triangle, so $xy$ is a rim edge of $q^\ast$.  We call $xy$ the
\emph{spine}; it is the spine of the diamond with tips $q^\ast$ and $z$.  The edges $xz$ and $yz$
are not in $R$: a rim edge $xz$ or $yz$ of some $q\in Q$ would make $q$ adjacent to two vertices
of $T$, so either $q\ne q^\ast$ would be a second double-boundary vertex of $Q$, or $q=q^\ast$
would be adjacent to all of $T$.  The vertex $z$ may or may not be a vertex of $R$.

\begin{lemma}\label{lem:budget}
Put $A_x=D-1-d_R(x)$ and $A_y=D-1-d_R(y)$.  Then $h=|N(x)\cap N(y)\cap O|$ satisfies
\begin{equation}\label{eq:hstar}
  h\le h^\ast:=\min\{D-4,\ A_x,\ A_y,\ \max(A_x,A_y)-1\},
\end{equation}
and in Lemma~\ref{lem:counts} it suffices to subtract $h^\ast$ from the two arms $q^\ast x$ and
$q^\ast y$, and from no other arm.
\end{lemma}

\begin{proof}
All vertices of $R$ lie inside or on $T$, being neighbors of interior vertices.  The edges of $R$
at $x$ and the edge $xz\notin R$ do not go to $O$, so $|N(x)\cap O|\le A_x$, and similarly for
$y$.  The edge $xy$ has the common neighbors $q^\ast$ and $z$, so Lemma~\ref{lem:nobook} gives
$h\le D-4$, and Lemma~\ref{lem:exterior} gives the rest of \eqref{eq:hstar}.  An arm $qx'$ has
some $t\in C(qx')$ with $x't\in E(T)$ only if $q$ is adjacent to two vertices of $T$, that is,
only if $q=q^\ast$ and $\{x',t\}=\{x,y\}$.
\end{proof}

At $D=6$ one has $d_R(x),d_R(y)\le4$, since $x$ and $y$ are adjacent to $z$ and, by
Lemma~\ref{lem:innermost}(iv), to vertices of $O$; so $A_x,A_y\ge1$ and $h^\ast\ge0$.

\begin{lemma}\label{lem:masks}
Fix the degrees, $\sigma$ and the spine $xy$.  At most two potential rim edges other than $xy$
meet $x$ or $y$.  Call a template a \emph{base} if it contains $xy$ and every potential rim edge
that avoids $x$ and $y$; so there are at most four bases.  Every template with spine $xy$ is
obtained from a base with the same $h^\ast$ by deleting rim edges that are not edges of $T$, so
Lemma~\ref{lem:monotone} applies with the subtraction $h^\ast$ at $q^\ast x$ and $q^\ast y$.
\end{lemma}

\begin{proof}
A middle vertex of one of the lists \eqref{eq:paths} lies on two potential rim edges, a split tip
on one, and a shared tip on one for each of its two vertices of $Q$.  Since $xy$ is one of these
at $x$ and at $y$, at most one other potential rim edge meets each of $x$ and $y$.  Given a
template with spine $xy$, take the base with the same rim edges at $x$ and $y$; the template
arises from it by deleting rim edges that avoid $x$ and $y$.  These are not edges of $T$, and
deleting them leaves $d_R(x)$, $d_R(y)$ and hence $h^\ast$ unchanged.  (Deleting a rim edge at $x$
or $y$ instead could raise $h^\ast$ and so lower a bound; this is why the bases include every
choice there.)
\end{proof}

For each degree triple in \eqref{eq:triples}, each $\sigma$, and each potential rim edge chosen as
the spine, let $r\le2$ be the number of other potential rim edges meeting the ends of the spine;
this gives $2^r$ bases.  In total there are $1{,}220$ bases.

\begin{lemma}\label{lem:onespine}
For each of the $1{,}220$ bases $R$ at $D=6$, the graph $H_R$ is colorable from all lists of
sizes at least the bounds of Lemma~\ref{lem:counts}, with $h^\ast$ subtracted as in
Lemma~\ref{lem:budget}.  This is checked by computer: for $971$ bases we give a removal order,
and for the other $249$ a nonzero coefficient (Appendix~\ref{app:comp}).  Consequently, at
$D=6$, if $Q$ has exactly one double-boundary vertex and a vertex of degree at most five, every
coloring of $G-Q$ extends to $G$.
\end{lemma}

\begin{proof}
Since $h\le h^\ast$, the actual lists are at least as large as these bounds.  By
Lemma~\ref{lem:masks}, colorability passes from the bases to every template with a spine, and
Lemma~\ref{lem:lists} gives the extension.
\end{proof}

\subsection{Proof for \texorpdfstring{$D=6$}{D=6}}

\begin{theorem}\label{thm:D6}
Theorem~\ref{thm:main} holds for $D=6$.
\end{theorem}

\begin{proof}
Suppose a minimal counterexample exists; Corollary~\ref{cor:structure} applies.  Apply the face
rule.  If every eligible vertex is nonnegative, then
without a separating triangle all charges are nonnegative, contradicting $-12$.  With one, the
hypotheses of Lemma~\ref{lem:equality} hold: faces retain $b_f(2-6/\ell)$, vertices with at most
one neighbor on $T$ are eligible and hence nonnegative, and a double-boundary vertex has charge
at least $d(v)-6\ge-2$, with equality only for degree four and triangular incident faces.  The
octahedral cap then contradicts minimality.

Otherwise let $v$ be a negative eligible vertex.  By Lemma~\ref{lem:selection} there is a facial
triangle $Q$ containing $v$, inside $T$ if $T$ exists, with at most one double-boundary vertex.
By Lemma~\ref{lem:normal}, $Q$ is described by its template.  Its degrees lie in $\{4,5,6\}$ with
$d(v)\le5$, so after renaming they form one of the triples \eqref{eq:triples}.  Some rim edge of
$q\in Q$ is an edge of $T$ exactly when $q$ is double-boundary, so the template of $Q$ is generic
or has exactly one spine.  Lemma~\ref{lem:generic} or Lemma~\ref{lem:onespine} then extends a
coloring of $G-Q$, contradicting minimality.
\end{proof}

\section{Proof of Theorem~\ref{thm:main} and remarks}\label{sec:end}

\begin{proof}[Proof of Theorem~\ref{thm:main}]
The cases $D\ge9$, $D=8$, $D=7$ and $D=6$ are Theorems~\ref{thm:D9}, \ref{thm:D8}, \ref{thm:D7}
and~\ref{thm:D6}.
\end{proof}

\emph{Degree six.}  The argument of Lemma~\ref{lem:H4} does not carry over to $D=6$, even in
graphs with all the properties of Corollary~\ref{cor:structure}.  Take a
center $v$, vertices $u_i,w_i,z_i$ ($i\in\mathbb Z/4$) and the edges $vu_i$, $u_iu_{i+1}$,
$u_iw_i$, $u_iw_{i-1}$, $u_iz_i$, $z_iw_i$, $z_iw_{i-1}$; they triangulate a disk bounded by the
cycle $z_0w_0z_1w_1z_2w_2z_3w_3$.  Gluing two copies of this disk along that cycle gives a
$3$-connected triangulation $G$ of the sphere with $18$ vertices, $\Delta(G)=6$, minimum degree
four and no separating triangle.  There is an $11$-color $D$-coloring of $G-v$ that does not
extend after the rim of $v$ is uncolored: the opposite rim lists are disjoint and all eight lists
lie in seven colors.  (The script \texttt{fourwheel\_d6\_example.py} supplied with this paper
checks these claims.)  The graph $G$ itself has a $D$-coloring, even with ten colors, so the
example rules out only this reduction.

\emph{List versions.}  The computer checks in Section~\ref{sec:d6} apply to arbitrary lists, but
several steps use a common palette: the separator gluing of Lemma~\ref{lem:3conn}, the explicit
colorings in Lemmas~\ref{lem:nobook} and~\ref{lem:deg3}, the color exchange in
Lemma~\ref{lem:deg3}, and the comparisons of lists coming from one coloring in
Lemmas~\ref{lem:deg3easy}, \ref{lem:deg3} and~\ref{lem:H4}.  We do not know whether the
$(2D-1)$-list version holds.

\emph{Other surfaces.}  The proof uses that triangles bound disks and that the global charge is
negative; both can fail on other surfaces.

\appendix

\section{Computational verification}\label{app:comp}

This appendix describes the computer checks behind Lemmas~\ref{lem:generic}
and~\ref{lem:onespine}, all at $D=6$.

\subsection*{A.1. The templates}
A template is given by a degree triple from \eqref{eq:triples}, the vector $\sigma$, encoded as
$\sigma_{ab}+2\sigma_{ac}+4\sigma_{bc}$, and the set of rim edges present, encoded as a bit mask
over the potential rim edges taken along the lists \eqref{eq:paths}, first for $a$, then $b$,
then $c$.  The vertices are named $a,b,c$, then $a0,a1,\dots$ for the middle vertices of the list
of $a$ (similarly for $b,c$), and $\mathrm{pab}$ for a shared tip or
$\mathrm{pab\_a},\mathrm{pab\_b}$ for split tips (similarly for $ac,bc$).  The edges of $U$ are
ordered lexicographically by the names of their ends, the smaller name first; this is also the
order of the variables of the graph polynomial.  Another order, with each exponent kept on the
same edge, would change only the sign of the coefficient.
\begin{itemize}
  \item \emph{Generic templates.}  For each triple in \eqref{eq:triples} and each $\sigma$, build
        the full template $R$, the ordered set $U$, the graph $H_R$, and the bounds $f$ of
        Lemma~\ref{lem:counts}.  This gives $72$ templates.
  \item \emph{Bases.}  For each triple, each $\sigma$ and each potential rim edge $xy$ as spine,
        list the (at most two) other potential rim edges meeting $x$ or $y$; for each subset,
        delete exactly that subset from the full template; compute $d_R(x),d_R(y)$ and $h^\ast$
        from \eqref{eq:hstar}; compute $U$, $H_R$ and $f$, subtracting $h^\ast$ at the arms
        $q^\ast x$ and $q^\ast y$, where $q^\ast$ is the vertex of $Q$ whose rim edge is $xy$.
        This gives $\sum 2^{r}=1{,}220$ bases.
\end{itemize}
Some of these templates may not occur in any plane graph; checking them does no harm.

\subsection*{A.2. Certificates}
For a graph $H$ on $\{0,\dots,n-1\}$ with bounds $f$, a \emph{certificate} that $H$ is colorable
from all lists of sizes at least $f$ is one of the following:
\begin{itemize}
  \item a \emph{greedy certificate}: a removal order $v_1,\dots,v_n$ with
        $|N_H(v_i)\cap\{v_{i+1},\dots,v_n\}|<f(v_i)$ for all $i$ (Lemma~\ref{lem:greedy});
  \item a \emph{CN certificate} (for Combinatorial Nullstellensatz): a vector $t$ with
        $0\le t_i<f_i$ and $\sum_i t_i=|E(H)|$, together with the nonzero integer
        $c=[X^t]\prod_{ij\in E(H),\,i<j}(X_i-X_j)$ (Lemma~\ref{lem:cn}).
\end{itemize}
Table~\ref{tab:counts} gives the numbers of certificates of each kind.

\begin{table}[htb]
\caption{Numbers of certificates.}
\label{tab:counts}
\centering
\begin{tabular}{lrrrrrr}
\toprule
 & \multicolumn{3}{c}{Generic} & \multicolumn{3}{c}{Bases}\\
\cmidrule(lr){2-4}\cmidrule(lr){5-7}
Degrees & Greedy & CN & Total & Greedy & CN & Total\\
\midrule
444 & 8 & 0 & 8 & 54 & 0 & 54\\
445 & 8 & 0 & 8 & 84 & 0 & 84\\
446 & 7 & 1 & 8 & 109 & 7 & 116\\
455 & 7 & 1 & 8 & 105 & 9 & 114\\
456 & 7 & 1 & 8 & 130 & 16 & 146\\
466 & 6 & 2 & 8 & 143 & 35 & 178\\
555 & 7 & 1 & 8 & 120 & 24 & 144\\
556 & 5 & 3 & 8 & 120 & 56 & 176\\
566 & 4 & 4 & 8 & 106 & 102 & 208\\
\midrule
Total & 59 & 13 & 72 & 971 & 249 & 1220\\
\bottomrule
\end{tabular}
\end{table}

\subsection*{A.3. Verification}
For each template the verifier rebuilds $R$ from the degree triple, $\sigma$, the mask and the
spine, using code written from this paper; it recomputes $U$, $H_R$ and $f$, compares them with
the record, and checks that the records cover each expected template exactly once.  It then
checks each greedy certificate, and for each CN certificate the bounds on $t$ and the
coefficient.  The coefficient is computed exactly, over the integers, by three methods: a
frontier dynamic program over the linear factors; elimination of one vertex at a time via
\[
  c(H,t)=(-1)^{b+d-t_v}\sum_{S\subseteq N_H(v),\,|S|=d-t_v} c\bigl(H-v,\ t|_{V(H)\setminus\{v\}}-\mathbf 1_S\bigr),
\]
where $d=d_H(v)$ and $b=|\{u\in N_H(v):u<v\}|$, with $c=1$ for the empty graph and empty target and
$c=0$ whenever some target exponent is negative or $t_v>d$; and the Alon--Tarsi formula
$c=(-1)^{\beta}(\mathrm{EE}-\mathrm{EO})$, where for an orientation of $H$ with out-degrees $t$,
$\beta$ is the number of its arcs from a larger to a smaller vertex and $\mathrm{EE}$,
$\mathrm{EO}$ count its spanning Eulerian subdigraphs with an even and an odd number of arcs
($c=0$ if there is no such orientation).  The first two methods come from the
code that generated the certificates, and the third was written independently.  The three values
agree with each other and with the record.

\subsection*{A.4. Reproduction}
The certificates and the verifier are the ancillary files of the arXiv version of this paper,
arXiv:2610.33322.  The
records are \texttt{d6\_generic\_sharp.jsonl} ($72$ lines) and \texttt{d6\_one\_baseline.jsonl}
($1{,}220$ lines).  The verifier is \texttt{check\_records.py}, with \texttt{tmpl.py} for the
templates and \texttt{frontier.py}, \texttt{independent\_coefficient.py} and \texttt{atcoef.py} for
the three coefficient methods.  Run in the directory of these files, the command
\[
  \texttt{python3 check\_records.py}
\]
performs all the checks of A.3; it uses only the Python standard library and takes under a minute
on a laptop.  The script \texttt{fourwheel\_d6\_example.py} checks the example of
Section~\ref{sec:end}.

\section*{Use of AI tools}
The author used GPT-6 Pro (OpenAI) to develop the proof from the author's key ideas and to
generate the computer certificates, and Claude Opus 5.5 (Anthropic) to check the arguments and
the certificates.  Both tools helped prepare the manuscript.  The author has checked the entire
proof and takes full responsibility for its content.


\begin{thebibliography}{99}

\bibitem{Alon} N. Alon, Combinatorial Nullstellensatz, \emph{Combin. Probab. Comput.} 8 (1999)
7--29.

\bibitem{AlonTarsi} N. Alon and M. Tarsi, Colorings and orientations of graphs,
\emph{Combinatorica} 12 (1992) 125--134.

\bibitem{CranstonWest} D. W. Cranston and D. B. West, An introduction to the discharging method
via graph coloring, \emph{Discrete Math.} 340 (2017) 766--793.

\bibitem{Diestel} R. Diestel, \emph{Graph Theory}, 5th ed., Graduate Texts in Mathematics 173,
Springer, 2017.

\bibitem{Hall} P. Hall, On representatives of subsets, \emph{J. London Math. Soc.} 10 (1935)
26--30.

\bibitem{HKW} X. Hu, J. Kong and Y. Wang, $D$-coloring of planar graphs, arXiv:2607.14837, 2026.

\bibitem{TW} L. Tian and R. Wang, A stronger upper bound on the $D$-chromatic index,
arXiv:2609.01875, 2026.

\bibitem{Wang} R. Wang, Proper edge coloring with rainbow diamonds, arXiv:2606.06831, 2026.

\end{thebibliography}
\end{document}